\documentclass[reqno,12pt]{amsart}
\usepackage{amssymb,amsmath,amsthm,amsxtra,mathtools,mathrsfs,calc,nccmath,bm, color}
\usepackage{enumerate}
\usepackage[margin=1in]{geometry}
\usepackage{calc}
\usepackage{graphicx}
\usepackage{caption}
\usepackage{subcaption}

\def\Z{\mathbb{Z}}
\def\Q{\mathbb{Q}}
\def\R{\mathbb{R}}
\def\H{\mathbb{H}}

\def\C{\mathbb{C}}

\def\bh{\bm H}

\def\sQ{\mathcal{Q}}
\def\fQ{\mathscr{Q}}
\def\sM{\mathcal{M}}
\def\sW{\mathcal{W}}

\DeclareMathOperator{\im}{Im}
\DeclareMathOperator{\re}{Re}
\DeclareMathOperator{\spt}{spt}

\DeclareMathOperator{\Tr}{Tr}
\DeclareMathOperator{\Mp}{Mp}
\def\SL{{\rm SL}}
\def\PSL{{\rm PSL}}

\def\CT{{\rm CT}}

\newcommand{\pfrac}[2]{\left(\frac{#1}{#2}\right)}
\newcommand{\pmfrac}[2]{\left(\mfrac{#1}{#2}\right)}

\newcommand{\pMatrix}[4]{\left(\begin{matrix}#1 & #2 \\ #3 & #4\end{matrix}\right)}
\renewcommand{\pmatrix}[4]{\left(\begin{smallmatrix}#1 & #2 \\ #3 & #4\end{smallmatrix}\right)}
\renewcommand{\bar}[1]{\overline{#1}}

\DeclareMathOperator{\sgn}{sgn}

\newtheorem{theorem}{Theorem}
\newtheorem{lemma}[theorem]{Lemma}

\newtheorem{proposition}[theorem]{Proposition}

\theoremstyle{remark}

\numberwithin{equation}{section}

\begin{document}

\title{A polyharmonic Maass form of depth $3/2$ for $\SL_2(\Z)$}
\date{\today}

\author{Scott Ahlgren}
\address{Department of Mathematics\\
University of Illinois\\
Urbana, IL 61801}
\email{sahlgren@illinois.edu}

\author{Nickolas Andersen}
\address{Department of Mathematics\\
UCLA\\
Los Angeles, CA 90095}
\email{nandersen@math.ucla.edu}

\author{Detchat Samart}
\address{Department of Mathematics\\
Burapha University\\
Chonburi, 20131 Thailand}
\email{petesamart@gmail.com}
\subjclass[2010]{11F37, 11F30}

\thanks{The first author was  supported by a grant from the Simons Foundation (\#426145 to Scott Ahlgren).
The second author was supported by NSF grant DMS-1701638.}

\begin{abstract}
Duke, Imamo\=glu, and T\'oth constructed a polyharmonic Maass form of level $4$ whose Fourier coefficients
encode real quadratic class numbers. A more general construction of such forms was subsequently given by Bruinier, Funke, and Imamo\=glu.
Here we give a direct construction of such a form for the full modular group and study the properties of its coefficients.
We give interpretations of the  coefficients of the holomorphic parts of each of these polyharmonic Maass forms as inner products of certain weakly holomorphic modular forms
and harmonic Maass forms.  The coefficients of square index are particularly intractable; in order to address these, we develop various extensions of the
usual normalized Peterson inner product using a strategy of Bringmann, Ehlen and Diamantis.
\end{abstract}

\maketitle

\section{Introduction}
We begin by discussing a polyharmonic Maass form of level $4$.
For $n\geq 0$, let $H(n)$ denote the  Hurwitz class number.  We have $H(0)=-1/12$ and $H(n)=0$ for $n\equiv 1,2\pmod 4$.  Otherwise $H(n)$ is the number of positive definite
quadratic forms of discriminant $-n$, counted with multiplicity equal to the inverse of the order of their stabilizer in $\SL_2(\Z)$.

Zagier \cite{zagier-class-numbers} introduced the first example of what is known as a harmonic Maass form.
For $y>0$ let $\beta_k(y)$ denote the normalized incomplete gamma function
\begin{equation}\label{eq:betadef}
	\beta_k(y) := \frac{\Gamma(1-k, y)}{\Gamma(1-k)} = \frac{y^{1-k}}{\Gamma(1-k)} \int_{1}^{\infty} t^{-k} e^{- yt} \, dt.
\end{equation}
Zagier defined the function
\begin{equation}\label{eq:zagform}
	\widehat{\bm Z}_-(\tau) := \sum_{n\geq0} H(n) q^{n} + \frac{1}{8\pi\sqrt y} - \frac{1}{4}  \sum_{n\neq 0} |n| \beta_\frac32(4\pi n^2 y) q^{-n^2}
\end{equation}
(we use the notation $\widehat{\bm Z}_-$ to follow the notation of
Duke, Imamo\=glu, and T\'oth  \cite{DIT:CycleIntegrals}).
Here, and throughout, $\tau=x+iy$ and $q=e(\tau)=e^{2\pi i\tau}$.
Zagier showed that the function $\widehat{\bm Z}_-(\tau)$ transforms like a modular form of weight $3/2$ on $\Gamma_0(4)$  and that
\[
\xi_{\frac 32}\widehat{\bm Z}_-=-\frac1{16\pi}\theta,
\]
where
\begin{equation}
	\xi_k := 2iy^k \overline{\frac{\partial}{\partial \bar\tau}}
\end{equation}
and $\theta(\tau) := \sum_{n\in \Z} q^{n^2}$ is the usual theta function.

Suppose that $d$ is a non-square discriminant.  If $d<0$ then
let $\omega_d$ be half the number of roots of unity in $\Q(\sqrt d)$, and if $d>0$
let $\varepsilon_d$ be the fundamental unit in $\Q(\sqrt d)$.
Following \cite[\S 2]{Duke:2011a}, define the regulator
\begin{equation}
R(d):=\begin{cases} 2\,\pi \omega_d^{-1}\quad&\text{if $d<0$},\\
2\log \varepsilon_d\quad&\text{if $d>0$, $d\neq\square$},
\end{cases}
\end{equation}
and  the general Hurwitz function
\begin{equation}\label{eq:hurwitz}
h^*(d):=\frac1{2\pi} \sum_{\ell^2\mid d} R\pmfrac d{\ell^2}h\pmfrac d{\ell^2},
\end{equation}
where $h(d)$ is the class number.
Note that for $d<0$ we have $h^*(d)=H(|d|)$.
Define
\begin{equation}\label{eq:alphadef}
	\alpha(y) := \frac{\sqrt y}{4\pi}  \int_0^\infty e^{-\pi y t}t^{-\frac 12} \log(1+t) \, dt.
\end{equation}
Duke, Imamo\=glu, and T\'oth \cite[Theorem 4]{DIT:CycleIntegrals} (see also \cite[(4.2)]{Duke:2011a})  showed that there is a nonholomorphic modular form of weight $\frac 12$ on $\Gamma_0(4)$ whose Fourier expansion is
\begin{multline}\label{eq:zhatplus}
	\widehat{\bm Z}_+(\tau)  := \sum_{d>0, \ d\neq \square} \frac{h^*(d)}{\sqrt d} q^d  + \sum_{n>0}a(n^2)q^{n^2} + \frac{\sqrt y}{3} + \sum_{d<0} \frac{h^*(d)}{\sqrt{|d|}} \beta_\frac12(4\pi|d|y) q^{d} \\*
	- \mfrac 1{4\pi} \log y + \sum_{n\neq 0} \alpha(4n^2 y) q^{n^2} -\mfrac 1\pi \left( \mfrac{\zeta'(2)}{\zeta(2)} - \gamma +  \log 4 \right)
\end{multline}
and for which
\[
	\xi_\frac12 \widehat{\bm Z}_+=-2\widehat{\bm Z}_-.
\]
Here $\gamma$ denotes Euler's constant, and we have corrected the value of the constant term using (5.4) and (2.24) of \cite{DIT:CycleIntegrals}.

The form $\widehat{\bm Z}_+(\tau)$ is defined through a limit of Poincar\'e series \cite[(5.4)]{DIT:CycleIntegrals}. The coefficients $a(n^2)$ are particularly intractable since they correspond to poles of the Poincar\'e series, and they are not determined in \cite{DIT:CycleIntegrals}.

Bruinier, Funke, and Imamo\=glu \cite{Bruinier:2011} introduced a general regularized theta lift which lifts weak Maass forms of weight zero to polyharmonic Maass forms of weight $1/2$.  By polyharmonic we mean that the form is annihilated by repeated application of the operators
$\xi_k$; a precise definition is in  Section~\ref{sec:results}.
Applying this lift to the constant function~$1$ produces  a function ${\bm Z}(\tau)$ which differs from  $\widehat{\bm Z}_+(\tau)$
by a constant multiple of $\theta(\tau)$,
and which  provides an interpretation of the mysterious coefficients of square index.
After some computation using Theorem~4.2 and Remark~3.4 of \cite{Bruinier:2011}
one can describe this form in such a way that every
non-trivial coefficient has an interpretation in terms of the general Hurwitz function.

To state this result, we extend the definition of $R(d)$ by setting
\begin{equation}
R(d) := 2\log \sqrt d \qquad \text{if  $d=\square$},
\end{equation}
and we define $h^*(d)$ via \eqref{eq:hurwitz}.
Then the work of Bruinier, Funke, and Imamo\=glu implies the following.
We note that there are a few typos in \cite[Theorem~4.2]{Bruinier:2011}; details and a sketch of the computation which produces the following
 result are given in
Section~\ref{sec:DIT}  below.

\begin{theorem}\label{thm:BFI}
There is a polyharmonic Maass form of weight $1/2$ and depth $3/2$ on $\Gamma_0(4)$ whose Fourier expansion is
\begin{equation}\label{eq:BFI}
	{\bm Z}(\tau) := \sum_{d>0} \frac{h^*(d)}{\sqrt d} q^d  + \frac{\sqrt y}{3} + \sum_{d<0} \frac{h^*(d)}{\sqrt{|d|}} \beta_\frac12(4\pi|d|y) q^{d}
	  + \frac{\gamma - \log (16\pi y)}{4\pi} + \sum_{n\neq 0} \alpha(4n^2 y) q^{n^2}
\end{equation}
and for which
\[\xi_\frac12 \bm Z=-2\widehat{\bm Z}_-.\]
\end{theorem}

The main result of Duke, Imamo\=glu, and T\'oth \cite{Duke:2011a} gives an interpretation of the coefficients $h^*(d)$ of $\widehat{\bm Z}_+$ as regularized inner products in the case when $d>0$ is not a square.
To describe the result, we recall that for  each $d>0$ there exists a unique weight $\frac{3}{2}$ weakly holomorphic modular form $g_d$ on $\Gamma_0(4)$ of the form
\[g_d(\tau)= q^{-d}+\sum_{0\le n\equiv 0,3 (4)}B(d,n)q^n,\]
where the $B(d,n)$ are integers and
\[B(d, 0)=\begin{cases} -2 &\ \ \text{if $d=\square$},\\
0& \ \ \text{otherwise.}
\end{cases}
\]
Proposition~4.1 of \cite{Duke:2011a} gives the formula
\[\langle g_d, \widehat{\bm Z}_-\rangle_{\operatorname{reg}}=-\frac34\frac{h^*(d)}{\sqrt{d}}\qquad \text{if $d>0$ is not square}.\]
 Here $\langle \cdot, \cdot\rangle_{\operatorname{reg}}$ is the usual regularized inner product.  The  integral
defining this inner product does not converge when $d$ is square.

Motivated by recent work of Bringmann, Diamantis and Ehlen \cite{BDE} we introduce a natural inner product
$\langle \cdot, \cdot\rangle_4$  which extends $\langle \cdot, \cdot\rangle_{\operatorname{reg}}$ and which allows us to treat the case when $d$ is square. We give the precise definition in Section~\ref{sec:DIT}.  Letting
\[\delta_\square(d)=\begin{cases} 1&\ \  \text{if $d$ is  square},\\
 $0$&\ \ \text{otherwise},\end{cases}\]
we prove the following.
\begin{theorem}\label{thm:innprodlevel4}
For  every positive discriminant $d$ we have
\begin{equation*}
\langle g_d, \widehat{\bm Z}_-\rangle_4=
 -\frac{h^*(d)}{\sqrt{d}}+\delta_{\square}(d)\left(\frac{\gamma-\log 4\pi}{2\pi}\right).
\end{equation*}
\end{theorem}

Our main goal in this paper is to introduce and to study a polyharmonic Maass form analogous to ${\bm Z}(\tau)$ on the full modular group.
Let $p(n)$ denote the partition function, and let $\spt(n)$ denote the number of smallest parts in the partitions of $n$. This function has been the object
of much study (see, for example, \cite{aa-spt, Andrews:spt, Bringmann:duke,  FO:spt,Garvan:spt}, and the references in these papers).
Let $\chi_{12}$ denote the Kronecker character for $\Q(\sqrt 3)$.
If we define
\begin{equation}\label{eq:sdef}
	s(n) := \spt(n) + \mfrac{1}{12} (24n-1) p(n),
\end{equation}
then work of Bringmann \cite{Bringmann:duke} (see the next section for details) shows that, in analogy with \eqref{eq:zagform}, the generating function
\begin{equation}\label{eq:fdef}
	F(\tau) := \sum_{n=-1}^\infty s\pmfrac{n+1}{24} q^\frac n{24} - \mfrac{1}{2} \sum_{n=1}^\infty \chi_{12}(n) n \, \beta_{\frac 32}\pmfrac{\pi n^2y}6 q^{-\frac {n^2}{24}}
\end{equation}
is a harmonic Maass form of weight $3/2$ on $\SL_2(\Z)$ with a certain multiplier.
In particular we have
\begin{equation} \label{eq:xi-F-eta}
\qquad \xi_\frac32F=-\mfrac{\sqrt 6}{4\pi}\,\eta,
\end{equation}
where $\eta(\tau):=q^{\frac1{24}}\prod_{n\geq 1}(1-q^n)$ is Dedekind's eta function.
In analogy with \eqref{eq:zhatplus} and \eqref{eq:BFI} we introduce a   polyharmonic Maass form $\bh(\tau)$ in Theorem~\ref{thm:h} below with
\[\xi_\frac12 \bh=-2\sqrt 6F.\]
The non-trivial coefficients of positive index are given by traces of a certain modular function $f$ of level $6$
over geodesics on the modular curve. 
Those of negative index are given by the numbers $s(n)$, which, by recent work of the first two authors \cite{aa-spt}, satisfy the relation
\begin{equation} \label{eq:s-alg-0}
	12 \, s\pmfrac{1-n}{24} = \sum f(\tau_Q),
\end{equation}
where the sum is over quadratic points in the upper half plane (see \eqref{eq:s-alg} below for details).
Again, the terms of square index are  intractable due to  poles
in the Poincar\'e series.

In order to deduce \eqref{eq:s-alg-0}, the authors of \cite{aa-spt} used a theta lift of Bruinier-Funke \cite{bf-traces}.  In a similar way, the
theta lift of \cite{Bruinier:2011} could be used to deduce Theorem~\ref{thm:h} below.  Here we  compute the expansion directly
from the limit definition.
Of course, most of the difficulty comes from the coefficients of square index.
We remark that this leads  to a proof of the algebraic formula  \eqref{eq:s-alg-0} which does not involve the theta lift.
We also remark that a similar argument applied to a suitable modification of the limit definition \cite[(5.4)]{DIT:CycleIntegrals} of
$\widehat{\bm Z}_+(\tau)$  produces the expansion \eqref{eq:BFI} without recourse to the theta lift.  We give a brief
discussion in Section~\ref{sec:DIT}, in which we also prove Theorem~\ref{thm:innprodlevel4}.

In the next section we give some background and describe the form $\bh(\tau)$.
Section~\ref{sec:def} contains the limit definition of the form $\bh(\tau)$ as well as some technical results on convergence issues (which are somewhat subtle).
In Sections~\ref{sec:nonsquare} and \ref{sec:square-index}
we compute the coefficients of non-square and square index, respectively.

In Section~\ref{sec:inn-prod}, we define a regularized inner product $\langle \cdot,\cdot \rangle_1$ for forms on $\SL_2(\Z)$ and we prove an analogue of Theorem~\ref{thm:innprodlevel4} for the form $F$.   In particular, we define  a family $\{h_d\}$ of weakly holomorphic forms of weight $\frac 32$ 
and describe the inner products $\langle h_d,F\rangle_1$ in Theorem~\ref{T:RIP}.

In this paper we have decided to emphasize precision and (to the extent   possible) simplicity in order to highlight  the polyharmonic Maass forms $\bm Z$ and $\bm H$.
The computations are quite subtle already,
but it seems clear that similar results hold in more generality.

\section{A polyharmonic Maass form of depth $3/2$ for $\SL_2(\Z)$} \label{sec:results}

Let $\spt(n)$ denote the number of smallest parts in the partitions of $n$,
and let $s(n)$ and $F(\tau)$ be defined as in \eqref{eq:sdef} and \eqref{eq:fdef}.
Let $\eta(\tau)$ be the Dedekind eta function, and define the  multiplier $\chi$  by
\begin{equation}\label{eq:chidef} \eta(\gamma\tau)=\chi(\gamma)\sqrt{c\tau+d}\,\eta(\tau)\qquad \text{ for  } \gamma=\pmatrix abcd\in \SL_2(\Z).
\end{equation}

After work of Bringmann \cite{Bringmann:duke} (see  Section~3 of \cite{aa-spt}), we know that  $F(\tau)$ is a harmonic Maass form of weight $3/2$ on $\SL_2(\Z)$ with multiplier $\bar \chi$.
Let $E_{2k}(\tau)$ denote the Eisenstein series of weight $2k$ on $\SL_2(\Z)$, and
let $f(\tau)$ be the modular function on $\Gamma_0(6)$ given by
\begin{equation} \label{E:f}
	f(\tau) = \frac{1}{24} \frac{E_4(\tau) - 4E_4(2\tau) - 9E_4(3\tau) + 36E_4(6\tau)}{(\eta(\tau)\eta(2\tau)\eta(3\tau)\eta(6\tau))^2} = q^{-1} + 12 + 77q + \ldots.
\end{equation}
For $n\equiv 1\pmod{24}$ define
\[
	\sQ_{n}:= \left\{ ax^2+bxy+cy^2: b^2-4ac=n, \, 6\mid a>0, \, \text{ and } b\equiv 1\bmod 12 \right\}.
\]
The group
\begin{equation}\label{eq:gamma_def}
\Gamma := \Gamma_0(6)/\{\pm 1\}
\end{equation}
 acts on this set, and for squarefree $n$, the class number $h(n)$ is the size of $\Gamma\backslash \sQ_n$ (see \cite[Section I]{GKZ}).

If $0>n\equiv 1\pmod{24}$ then for each $Q\in \sQ_{n}$ let $\tau_Q$ denote the root of $Q(\tau,1)$ in the upper-half plane $\H$.
In \cite{aa-spt} (see the end of Section~3) the first two authors proved
 the algebraic formula
\begin{equation} \label{eq:s-alg}
	12 \, s\pmfrac{1-n}{24} = \Tr_{n}(f) := \sum_{Q\in \Gamma\backslash \sQ_{n}} \!\!\! f(\tau_Q).
\end{equation}
This together with some analytic considerations leads to a transcendental formula  for $\spt(n)$ (Theorem~1 of \cite{aa-spt})
which is analogous to Rademacher's formula for $p(n)$.

For an indefinite quadratic form $Q$ with non-square discriminant, let $C_Q$ denote the geodesic in the upper half plane
connecting the roots of $Q$, modulo the stabilizer of $Q$.
For positive non-square $n\equiv 1\pmod{24}$, define
\begin{equation} \label{eq:def-trace-non-square}
 	\Tr_n(f) := \mfrac{1}{2\pi} \sum_{Q\in \Gamma\backslash\sQ_n} \int_{C_Q} f(\tau) \, \frac{d\tau}{Q(\tau,1)}.
\end{equation}

The situation is more difficult when $n$ is  square.  In this case the geodesics $C_Q$ are infinite and the  integrals in \eqref{eq:def-trace-non-square} are divergent.
In analogy with  \cite{andersen-inf-geo,andersen-singular} we will define dampened versions of the function $f(\tau)$ in Section~\ref{sec:square-index}.
For each quadratic form with square discriminant we will construct a function  $f_Q(\tau)$ by subtracting off the constant term and the exponentially growing terms in the Fourier expansion of $f$ at the cusps corresponding to roots of $Q(\tau,1)$.  See  \eqref{eq:fqdef} for the precise definition.
For each indefinite form $Q$ with non-square discriminant we  define $f_Q(\tau):=f(\tau)$ to ease notation.
Then we can make the uniform definition
\begin{equation}\label{eq:fqtil}
	\Tr_n(f) := \mfrac 1{2\pi} \sum_{Q\in \Gamma\backslash\sQ_n} \int_{C_Q} f_Q(\tau) \, \frac{d\tau}{Q(\tau,1)}
\end{equation}
for  every positive $n\equiv 1\pmod {24}$.

We follow \cite{LR-polyharmonic} in defining polyharmonic Maass forms.\footnote{The polyharmonic Maass forms in \cite{LR-polyharmonic} have at most polynomial growth at the cusps. While it would be more precise to call the functions in this paper polyharmonic \emph{weak} Maass forms, we follow the usual practice~\cite{hmf-book} of dropping the adjective ``weak''.}
Let $\Delta_k$ denote the weight $k$ hyperbolic Laplacian
\begin{equation} \label{eq:def-deltak}
	\Delta_k := -\xi_{2-k} \xi_k = -y^2 \left( \mfrac{\partial^2}{\partial x^2} + \mfrac{\partial^2}{\partial y^2} \right) + i k y \left( \mfrac{\partial}{\partial x} + i \mfrac{\partial}{\partial y} \right).
\end{equation}
Let $m$ be a non-negative integer.
We say that a real analytic function $h:\H\to\C$ is a polyharmonic Maass form of weight $k$ and depth $m$ on $G\subseteq \SL_2(\Z)$ with multiplier $\nu$ if
\begin{equation*}
	\Delta_k^m h = 0
\end{equation*}
and if
\begin{equation}
	h(\gamma \tau) = \nu(\gamma) (c\tau+d)^k h(\tau) \qquad \text{ for all } \gamma=\pmatrix abcd\in G.
\end{equation}
We allow $h$ to have exponential growth at the cusps of $G$.
We define half-integral depths by assigning weakly holomorphic modular forms a depth of $\frac 12$ since they are annihilated by $\xi_k$, which is essentially ``half'' of $\Delta_k$.
We say that $h$ is polyharmonic  of depth $m+\frac 12$ if $\Delta_k^m h$ is weakly holomorphic.

Recall the definitions \eqref{eq:alphadef} \eqref{eq:betadef} and  \eqref{eq:chidef} of $\alpha$, $\beta_k$,  and $\chi$.
By \cite[\S8.2]{NIST:DLMF} we have
\begin{equation}\label{eq:beta_to_gamma}
\beta_k(y)=1-y^{1-k}\gamma^*(1-k, y),
\end{equation}
where $\gamma^*(1-k,z)$ is an entire function of $z$, given by
\begin{equation}\label{eq:gammastardef}
\gamma^*(1-k, z)=\frac1{\Gamma(1-k)}\int_0^1t^{-k}e^{-zt}\, dt\qquad \text{for $k<1$.}
\end{equation}
We use the principal branch of the square root, with the convention that
\[\sqrt{-y}=i\sqrt y\qquad \text{for $y>0$.}\]
This gives a definition of $\beta_\frac12(-y)$ for $y>0$.

In analogy with Theorem~\ref{thm:BFI} we will prove
\begin{theorem}\label{thm:h}
The function
\begin{multline} \label{eq:h-exp}
	\bh(\tau) = -i \, q^{\frac 1{24}} + \sum_{0<n\equiv 1(24)} \Tr_n(f) q^{\frac n{24}} + 12 \sum_{n\geq 1}\frac{\chi_{12}(n)}{n} h^*(n^2) q^{\frac{n^2}{24}}  \\*
	+ i \beta_{\frac 12}\pmfrac{-\pi y}6 q^{\frac 1{24}} +  \sum_{0>n\equiv 1(24)}   \, \frac{\Tr_n(f)}{\sqrt{|n|}}
	\beta_{\frac 12}\pmfrac{\pi |n|y}6 q^{\frac n{24}}
	  + 24 \sum_{n\geq 1} \chi_{12}(n) \, \alpha \left( \mfrac {n^2 y}6 \right)  q^{\frac{n^2}{24}}
\end{multline}
is a polyharmonic Maass form of weight $\frac 12$ and depth $\frac 32$ on $\SL_2(\Z)$ with multiplier $\chi$.
Moreover, we have $\xi_{\frac 12} \bh(\tau) = - 2\sqrt 6 F(\tau)$ and $\Delta_{\frac 12} \bh(\tau) = -\tfrac{3}{\pi} \eta(\tau)$.
\end{theorem}

\section{Definition of the polyharmonic Maass form $\bh(\tau)$} \label{sec:def}
In this section we define the polyharmonic Maass form $\bh(\tau)$ as the constant term in the Laurent expansion at $s=\frac 34$ of a Poincar\'e series $P(\tau,s)$.
This requires the Whittaker functions $M_{\kappa,\mu}(y)$ and $W_{\kappa,\mu}(y)$ (see \cite[\S 13.4]{NIST:DLMF}) and the Bessel functions $J_\nu(x)$, $I_\nu(x)$, and $K_\nu(x)$ (see \cite[Chapter 10]{NIST:DLMF}).
Let $\Gamma_\infty := \{\pm\pmatrix 1*01\}\subseteq \SL_2(\Z)$.
For $\re(s)>1$ we define the weight $\frac 12$ Poincar\'e series
\begin{equation*}
	P(\tau,s) := \mfrac 12 \sum_{\gamma \in \Gamma_\infty \backslash \SL_2(\Z)} \bar \chi(\gamma) (c\tau+d)^{-\frac 12} \sM(\im \gamma\tau,s)e_{24}(\re \gamma\tau),
\end{equation*}
where $e_m(x):=e\left(\frac xm\right)$ and
\begin{equation*}
	\sM(y,s) := \pmfrac{\pi y}{6}^{-\frac 14} M_{\frac 14, s-\frac 12} \pmfrac{\pi y}{6}.
\end{equation*}
Define the generalized Kloosterman sum
\begin{equation}\label{Acn}
A_c(n):=\sum_{\substack{d \bmod c\\ (d,c)=1}}e^{\pi i s(d,c)}e \left( -\frac{dn}c \right) ,
\end{equation}
where $s(d,c)$ is the Dedekind sum
\begin{equation*}
	s(d,c) := \sum_{r=1}^{c-1} \frac rc \left( \frac{dr}c - \left\lfloor \frac{dr}c \right\rfloor - \frac 12 \right).
\end{equation*}
The papers \cite{aa-52} and \cite{andersen-singular} used the notation
\begin{equation}
	K(m,n;c) := \sum_{d\bmod c^*} e^{\pi i s(d,c)} e\pfrac{\bar{d}m+dn}{c};
\end{equation}
we have
\begin{equation}\label{eq:kloos_reconcile}
A_c(n)=K(0, -n, c).
\end{equation}
 By Proposition 8 of \cite{aa-52} we have the Fourier expansion
\begin{equation*}
	P(\tau,s) = \sM(y,s) e_{24}(x)
	+ \sum_{n\equiv 1(24)} a(n,s) \sW_n(y,s) e_{24}(nx)
\end{equation*}
where
\begin{equation*}
	\sW_{n}(y,s) := \pmfrac{\pi|n|y}{6}^{-\frac 14} W_{\frac {\sgn n}4,s-\frac 12} \pmfrac{\pi|n|y}{6}
\end{equation*}
and
\begin{equation}\label{eq:ansdef}
	a(n,s) = \frac{2\pi \Gamma(2s)}{|n|^{\frac 14}\Gamma(s+\frac{\sgn n}4)} \sum_{c>0} \frac{A_c\pfrac{1-n}{24}}{c}\times\begin{cases}
	 J_{2s-1}\pmfrac{\pi \sqrt n}{6c}\ &\text{if $n>0$},\\
	I_{2s-1}\pmfrac{\pi \sqrt {|n|}}{6c}\ &\text{if $n<0$}.
	\end{cases}
\end{equation}

Lehmer \cite[Theorem~8]{Lehmer:series} proved the sharp Weil-type bound
$|A_c(n)|\leq 2^{\omega_0(c)} \sqrt c$, where $\omega_0(c)$ is the number of distinct odd primes dividing $c$.
Using this together with the estimates \eqref{eq:whittakerest}, \eqref{eq:jbesselbound}, \eqref{eq:IBesselbound1} and \eqref{eq:IBesselbound2} below, one can show that $P(\tau,s)$ has an analytic continuation to $\re(s)>3/4$  (this can also be seen using  the estimates in Lemma~\ref{lem:ans} below).

Let
\begin{equation}\label{eq:csdef}
	c(s) := \frac{(2s-1)\Gamma(s-\frac 14)\Gamma(2s-\frac 12)(2^{2s-\frac 12}-1)(3^{2s-\frac 12}-1)\zeta(4s-1)}{6^{s-\frac 34}\pi^{2s}\Gamma(2s)}.
\end{equation}
We define $\bh(\tau)$ as follows:
\begin{equation}\label{eq:hdef}
	\bh(\tau) := \frac{6}{\pi} \lim_{s\to\frac 34^+} \left( c(s) P(\tau,s) - \frac{\eta(\tau)}{(s-\frac 34)} \right).
\end{equation}
Let $\chi_{12}$ denote the Kronecker character for $\Q(\sqrt 3)$ (with the convention that $\chi_{12}(n):=0$ if $n\not\in\Z$).
The goal of this section is to prove the following.
\begin{proposition}\label{prop:hdef}
\begin{enumerate}
\item The limit defining $\bh(\tau)$ exists.
\item
  The function $\bh(\tau)$   is a polyharmonic Maass form of weight $\frac 12$ and depth $\frac 32$ on $\SL_2(\Z)$ with multiplier $\chi$.
\item We have $\xi_\frac12 \bh(\tau)=-2\sqrt 6 F(\tau)$ and $\Delta_\frac12 \bh(\tau)=-\frac{3}\pi\eta(\tau)$.
\item We have
\begin{multline}\label{eq:htau1}
	\bh(\tau) = -i  \left( 1- \beta_{\frac 12}\pmfrac{-\pi y}6 \right)  \, q^{\frac 1{24}} +  \mfrac{2}{\sqrt \pi}\sum_{\substack{0<n\equiv 1(24) \\ n\neq\square}}a(n,\tfrac34)q^{\frac{n}{24}} + 2 \!\! \sum_{0>n\equiv 1(24)}a(n,\tfrac34)\beta_{\frac 12}\pmfrac{\pi |n|y}6 q^{\frac{n}{24}}
	 \\ + \mfrac{6}{\pi} \sum_{\substack{0<n\equiv 1(24) \\ n=\square}} \lim_{s\to\frac 34} \left(c(s)a(n,s) \sW_n(y,s)- \mfrac{\chi_{12}(\sqrt n)}{s-\frac 34} e^{-\frac{\pi n y}{12}}\right)e_{24}(nx).
\end{multline}
\end{enumerate}
\end{proposition}

We require the following lemma, whose proof we leave to the end of this section.
\begin{lemma}\label{lem:ans}
There exists a positive constant $A$ such that the following are true.
\begin{enumerate}
\item
For $n<0$ we have
\begin{equation} \label{eq:b-n-s-3/4}
	a(n,s) = O \left( |n|^A \exp \left( \mfrac{\pi\sqrt{|n|}}6 \right)  \right) 
\end{equation}
uniformly for $s\in [\frac34, 1]$.
\item For $n>0$ we have
\begin{equation} \label{eq:a-n-s-3/4}
	a(n,s) =\chi_{12}(\sqrt n)\frac{3}{\sqrt\pi}\cdot \frac{1}{s-\frac 34} + O(n^A)
\end{equation}
uniformly for $s\in(\frac 34,1]$ if $n$ is square, and for $s\in [\frac34, 1]$ if $n$ is not square.
\end{enumerate}
\end{lemma}

For the remainder of the section we   let $A$   denote a positive constant whose value is
allowed to change at each occurrence.  We assume that $\tau$ is in a fixed compact subset of the upper half plane
(in particular, that $y$ is bounded away from $0$).
The implied constants in the estimates which follow will depend on the particular choice of compact subset.

This estimate can be derived  from the sentence which follows (13.19.3) of   \cite{NIST:DLMF}:
\begin{equation}\label{eq:whittakerest}
\sW_{n}(y,s)\ll |n|^A\exp \left( -\mfrac{\pi |n| y}{12} \right) ,\ \ \ s\in [\tfrac34,1].
\end{equation}

\begin{proof}[Proof of Proposition~\ref{prop:hdef}]
Using \cite[(13.14.31), (13.18.2), (13.18.5), and (13.14.32)]{NIST:DLMF}, we have the evaluations
\begin{align}
	\sW_n\left(y,\mfrac 34\right) &=
	\begin{cases}
		e^{-\frac{\pi n y}{12}} & \text{ if }n>0, \\
		\sqrt{\pi}\beta_{\frac 12}\pmfrac{\pi |n|y}6  e^{-\frac{\pi n y}{12}} &\text{ if }n<0,
	\end{cases} \\
	\sM\left(y,\mfrac 34\right) &= -\mfrac{i\sqrt \pi}{2} \left(1- \beta_{\frac 12}\pmfrac{-\pi y}6\right) e^{-\frac{\pi y}{12}}.
\end{align}
We also have
\begin{equation}\label{eq:c34}
	c\pmfrac 34=\frac{\sqrt \pi}3.
\end{equation}
This gives the first term  in \eqref{eq:htau1}.
If  $n$ is not square then Lemma~\ref{lem:ans} and  \eqref{eq:whittakerest} give
\[a(n,s) \sW_{n}(y,s)\ll |n|^A \exp \left( -\mfrac{\pi |n| y}{12}+\mfrac{\pi \sqrt{|n|}}6 \right) ,\ \ \ s\in [\tfrac34,1].\]
This justifies the interchange which gives the first  two sums in \eqref{eq:htau1}.

For $n>0$, a straightforward computation using the integral representation  \cite[(9.222.1)]{GR-table}
\begin{equation}\label{eq:whittaker_integral}
	W_{\frac 14,s-\frac 12}(y) = \frac{y^s e^{-\frac y2}}{\Gamma(s-\frac 14)} \int_0^\infty e^{-yt} t^{s-\frac 54} (1+t)^{s-\frac 34} \, dt
\end{equation}
shows that  $\frac{\partial}{\partial s} \sW_n(y,s)$ also satisfies the bound \eqref{eq:whittakerest}.
Expanding at $s=3/4$ using this fact together with \eqref{eq:a-n-s-3/4},  \eqref{eq:whittakerest} and \eqref{eq:c34} gives
\[c(s)a(n,s) \sW_n(y,s)- \frac{\chi_{12}(\sqrt n)}{s-\frac 34} e^{-\frac{\pi n y}{12}}\ll  n^A\exp \left( \mfrac{-\pi n y}{12}  \right) ,\ \ \ s\in (\tfrac34,1].\]
This justifies the interchange in the sum containing the square coefficients,  and  gives assertions (1) and (4).

The transformation properties of $\bh(\tau)$
are inherited from those of $P(\tau, s)$ and $\eta(\tau)$.
From (1) we have
\[\lim_{s\to\frac34^+}(s-\tfrac 34)c(s)P(\tau, s)=\eta(\tau).\]
A computation involving \cite[(13.14.1)]{NIST:DLMF} shows that $\Delta_\frac12 P(\tau, s)=(s-\frac 14)(\frac 34-s) P(\tau, s)$; from this we conclude that $\Delta_\frac12 \bh(\tau)=-\frac{3}\pi \eta(\tau)$
(to interchange the limit and the derivatives  requires uniform convergence, which follows as above).
From \eqref{eq:def-deltak} and \eqref{eq:xi-F-eta} we have
\[\xi_\frac32 \left( \xi_\frac12 \bh(\tau) \right) =\xi_\frac32 \left( -2\sqrt 6 F(\tau) \right) .\]
Now $\xi_\frac12 \bh(\tau)$ is a weak harmonic Maass form of weight $3/2$ and multiplier $\overline\chi$ on $\SL_2(\Z)$.
Using \eqref{eq:beta_to_gamma} and \eqref{eq:gammastardef} we find that the only exponentially growing term in its expansion is
\[\xi_\frac12 \left( i\beta_{\frac 12}\pmfrac{-\pi y}6q^\frac1{24} \right) =\frac{1}{\sqrt 6} \,q^{-\frac1{24}}.\]
From \eqref{eq:fdef} we conclude that $\xi_\frac12 \bh(\tau)=-2\sqrt 6F(\tau)$, since there are no modular forms of weight  $3/2$ on $\SL_2(\Z)$ with multiplier $\overline\chi$ which
are holomorphic both on $\mathbb H$ and at $\infty$.
This finishes the proof of the proposition.
\end{proof}

We turn to the proof of Lemma~\ref{lem:ans}.
\begin{proof}[Proof of Lemma~\ref{lem:ans}]
Define
\[S(n, x):=\sum_{c\leq x}\frac{A_c\pfrac{1-n}{24}}{c}.\]
By Theorem 3 of \cite{aa-spt},
we have the asymptotic formula
\begin{equation*}
	S(n,x) = \chi_{12}(\sqrt n) \mfrac{12\sqrt 3}{\pi^2} x^{\frac 12} + O_n(x^{\frac 16+\epsilon})
\end{equation*}
for any $\epsilon>0$.
While the $n$-dependence in the error term is not given explicitly in \cite{aa-spt}, a straightforward modification of the proof (following arguments given in, e.g., \cite{pribitkin}) shows that error term depends at worst polynomially on $n$.
Taking $\epsilon=\frac 1{12}$, we conclude that
\begin{equation}\label{eq:snx}
	S(n,x) = \chi_{12}(\sqrt n) \mfrac{12\sqrt 3}{\pi^2} x^{\frac 12} + O(|n|^A x^{\frac 14}).
\end{equation}

Suppose first that $n>0$. We require the facts that
\begin{equation}\label{eq:jbesselbound}
|J_\nu(y)|\leq \frac{|\frac12y|^\nu}{\Gamma(\nu+1)},\qquad \nu \geq -\tfrac12,
\end{equation}
\begin{equation}\label{eq:jprime}
J_\nu'(y)=-J_{\nu+1}(y)+\frac{\nu}y J _\nu(y),
\end{equation}
\begin{equation}\label{eq:jpos}
|J_\nu(y)|\leq 1,\qquad  y\in \R, \  \nu\geq0,
\end{equation}
and that $J_\nu(y)$ decays at $\infty$
\cite[(10.14.4), (10.6.2), (10.14.1), (10.17.3)]{NIST:DLMF}.
Combining  \eqref{eq:jbesselbound}  and \eqref{eq:jprime},
we obtain, for $t\geq 1$,
\begin{equation}\label{eq:jderiv_est}
\left[J_{2s-1}\pmfrac{\pi \sqrt n}{6t}\right]'\ll\frac{n^{A}}{t^{2s}}.
\end{equation}
Then, for  $s\in (3/4, 1]$,
partial summation together with  \eqref{eq:jbesselbound}  and \eqref{eq:jprime} gives
\begin{equation}\label{eq:ACsquare}
\sum_{c>0} \frac{A_c\pfrac{1-n}{24}}{c} J_{2s-1}\pmfrac{\pi \sqrt n}{6c}=
-\int_1^\infty S(n,t)\left[J_{2s-1}\pmfrac{\pi \sqrt n}{6t}\right]'\, dt.
\end{equation}
We first consider the case when $n$ is square.
From \eqref{eq:ansdef},  \eqref{eq:ACsquare} and \eqref{eq:snx} we obtain
\begin{equation*}
\begin{aligned}
a(n,s)&=  -\chi_{12}(\sqrt n) \frac{24\sqrt 3}{n^\frac14\pi} \frac{ \Gamma(2s)}{\Gamma(s+\frac14)}
\int_1^\infty \sqrt t \left[J_{2s-1}\pmfrac{\pi \sqrt n}{6t}\right]'\, dt+O(n^A).
\end{aligned}
\end{equation*}
Integrating by parts using \eqref{eq:jbesselbound} and \eqref{eq:jpos}, and then using the evaluation \cite[(10.22.43)]{NIST:DLMF}, the integral in the last line becomes
\begin{equation*}\label{eq:intevaluate}
 \begin{aligned}
-\mfrac12\int_1^\infty t^{-\frac12}\,  J_{2s-1}\pmfrac{\pi \sqrt n}{6t}\, dt+O(1)
&=-\mfrac12\int_0^\infty t^{-\frac12}\,  J_{2s-1}\pmfrac{\pi \sqrt n}{6t}\, dt+O(1)\\
&=-\mfrac 18 \sqrt{\mfrac \pi 3} \, n^{\frac 14} \frac{\Gamma(s-\frac 34)}{\Gamma(s+\frac 34)}+O(1),
\end{aligned}
\end{equation*}
from which we obtain
\[
a(n,s)=  \chi_{12}(\sqrt n)\frac{3}{\sqrt \pi} \cdot \frac{ \Gamma(2s)\Gamma(s-\frac34)}{\Gamma(s+\frac14)\Gamma(s+\frac34)}+O(n^A).\]
From the expansion  $\frac{ \Gamma(2s)\Gamma(s-\frac34)}{\Gamma(s+\frac14)\Gamma(s+\frac34)}=\frac1{s-3/4}+O(1)$
for $s\in (\frac34, 1]$,
we obtain
 \eqref{eq:a-n-s-3/4} in this case.
If $n$ is not square, then \eqref{eq:ACsquare} holds for $s\in[3/4,1]$, and the result follows from \eqref{eq:snx}.

The case when $n<0$  is similar.  We require the facts that for fixed $\nu>0$ we have
\begin{equation}\label{eq:IBesselbound1}
I_\nu(x)\ll\frac {x^\nu}{\Gamma(\nu+1)}\qquad \text{as $x\to 0$},
\end{equation}
 that
\begin{equation}\label{eq:IBesselbound2}
I_\nu(x)  \ll\frac {e^x}{\sqrt{x}}\qquad \text{as $x\to \infty$},
\end{equation}
that
\[I_\nu'(x)=I_{\nu+1}(x)+\frac{\nu}x I _\nu(x),\]
and that $I_\nu(x)$ is increasing as a function of  $x$
\cite[(10.30.1), (10.30.4), (10.29.2), (10.37)]{NIST:DLMF}.
We break the integral analogous to \eqref{eq:ACsquare} at $t=\sqrt{|n|}$.
The first part of the integral is $O\big(|n|^A \exp(\pi\sqrt{|n|}/6)\big)$.
Using  \eqref{eq:IBesselbound1} we find that
the second part is $O \left( |n|^A \right) $.  This gives \eqref{eq:b-n-s-3/4}.
\end{proof}

\section{Poincar\'e series and the coefficients of nonsquare index} \label{sec:nonsquare}

In this section we relate the coefficients $a(n,\frac 34)$ of non-square index appearing in Proposition~\ref{prop:hdef} to the traces $\Tr_n(f)$ defined in \eqref{eq:def-trace-non-square}.
We first define a function $f(\tau,s)$ in terms of a Maass-Poincar\'e series which, analytically continued, specializes to $f(\tau)$ at $s=1$.
For $r\mid 6$, define the Atkin-Lehner matrix $W_r$ by
\[
	W_1 = \pMatrix 1001, \, W_2 =\frac 1{\sqrt 2} \pMatrix2{-1}6{-2}, \, W_3 = \frac 1{\sqrt 3} \pMatrix3163, \, W_6 = \frac1{\sqrt 6} \pMatrix 0{-1}60.
\]
Then
\begin{equation} \label{eq:atkincompose}
W_d W_{d'}=W_\frac{d d'}{(d,d')^2}.
\end{equation}
Recall that $\Gamma=\Gamma_0(6)/\{\pm1\}$.
Following Section~4 of \cite{aa-spt} we define
\begin{equation}
	f(\tau,s) := \sum_{r\mid 6} \mu(r) \sum_{\gamma \in \Gamma_\infty \backslash \Gamma} \phi_s(\im \gamma W_r \tau) e(-\re\gamma W_r\tau), \qquad \re(s)>1,
\end{equation}
where $\mu$ is the M\"obius function and
\begin{equation*}
	\phi_s(y) := 2\pi \sqrt y I_{s-\frac 12}(2\pi y).
\end{equation*}
The function $f(\tau,s)$
satisfies
\begin{equation*}
	f(\gamma \tau,s) = f(\tau,s) \qquad \text{ for all }\gamma\in \Gamma_0(6)
\end{equation*}
and
\begin{equation}
	f(W_r\tau,s) = \mu(r) f(\tau,s) \qquad \text{ for } r\mid 6.
\end{equation}
We also have
\begin{equation}
	\Delta_0 f(\tau,s) = s(1-s) f(\tau,s).
\end{equation}
As shown in \cite[\S 4]{aa-spt}, the function $f(\tau,s)$ has an analytic continuation to $\re(s)>\frac 34$ and
\[
	f(\tau,1)=f(\tau).
\]

In the next section we will need the Fourier expansion of $f(\tau,s)$.
By Proposition 5 of \cite{aa-spt} we have
\begin{equation} \label{eq:f-tau-s-fourier}
	f(\tau,s) = 2\pi\sqrt y I_{s-\frac 12}(2\pi y) e(-x) + a_s(0) y^{1-s} + 2\sqrt y \sum_{n\neq 0} a_s(n) K_{s-\frac 12}(2\pi|n|y) e(nx),
\end{equation}
where
\begin{equation}\label{eq:aszero}
	a_s(0) = \frac{2\pi^{s+1}}{(s-\frac 12)\Gamma(s)} \sum_{r\mid 6} \mu(r) \sum_{\substack{0<c\equiv 0(6/r) \\ (c,r)=1}} \frac{k(-\bar r,0;c)}{(c\sqrt r)^{2s}}.
\end{equation}
Here $k(a,b;c)$ is the ordinary Kloosterman sum.
Exact formulas for the coefficients $a_s(n)$, $n\neq 0$, are given in \cite{aa-spt}, but we will need only  the crude estimate
\[
	a_s(n) \ll e^{6\pi\sqrt n} \qquad \text{ uniformly for } s\in [1, \tfrac 32],
\]
from \cite[(4.5)]{aa-spt}.
The constant coefficient $a_s(0)$ simplifies in the following way.
\begin{lemma} \label{eq:as0-simple}
For $\re(s)>\frac 12$ we have
\begin{equation} \label{eq:as0}
	a_s(0) = \frac{4\pi^{s+1}}{(2s-1)\Gamma(s)(2^s-1)(3^s-1)\zeta(2s)}.
\end{equation}
\end{lemma}
\begin{proof}
Evaluating the Kloosterman sums  in \eqref{eq:aszero} gives
\begin{equation}
\sum_{\substack{0<c\equiv 0(6/r) \\ (c,r)=1}} \frac{k(-\bar r,0;c)}{(c\sqrt r)^{2s}} = \frac{\mu(\frac6r)r^s}{6^{2s}} \sum_{\substack{c>0 \\ (c,6)=1}} \frac{\mu(c)}{c^{2s}},
\end{equation}
where we have replaced $c$ by $\frac{6c}r$ and used the fact that $\mu(\frac{6c}r)=0$ unless $(c,\frac6r)=1$.
Therefore
\[
\sum_{r\mid 6} \mu(r) \sum_{\substack{0<c\equiv 0(6/r) \\ (c,r)=1}} \frac{k(-\bar r,0;c)}{(c\sqrt r)^{2s}} = \frac{1}{6^{2s}} \sum_{r\mid 6} r^s \sum_{\substack{c>0 \\ (c,6)=1}} \frac{\mu(c)}{c^{2s}} = \frac1{(2^s-1)(3^s-1)\zeta(2s)},
\]
and the lemma follows.
\end{proof}

We will write the coefficients $a(n,\frac 34)$ in terms of the traces $\Tr_n(f)$ by modifying the proof of Proposition~7 of \cite{andersen-singular}.
We begin by noting that $f(\tau,s)$ is related to the function $P_1(\tau,s)$ defined in (2.7) of that paper by
\[
	P_1(\tau,2s-\tfrac 12) = \frac{2\Gamma(s+\frac 14)}{\sqrt{\pi} \Gamma(s-\frac 14)} f(\tau,2s-\tfrac 12).
\]
Suppose first that $n<0$.
Following the proof of \cite[Proposition~7]{andersen-singular}, we find that
\begin{equation*}
	|n|^{-\frac 14} \sum_{Q\in \Gamma\backslash \sQ_n} P_1(\tau_Q,s) = \frac{2\sqrt{2\pi}\Gamma(\frac{s+1}{2})}{\Gamma(\frac s2)} |n|^{-\frac 14} \sum_{\substack{Q\in \Gamma_\infty \backslash \sQ_n \\ Q=[a,b,c]}} \pmfrac{12}{b} a^{-\frac 12} I_{s-\frac 12}\pfrac{\pi\sqrt {|n|}}{a} e\pmfrac{b}{2a}.
\end{equation*}
Then the argument which follows \cite[(5.15)]{andersen-singular} shows that
\begin{equation*}
	|n|^{-\frac 14} \sum_{Q\in \Gamma\backslash \sQ_n} P_1(\tau_Q,s) = \frac{4\sqrt{\pi} \,\Gamma(\frac{s+1}{2})}{\Gamma(\frac s2)} |n|^{-\frac 14} \sum_{c>0} \frac{A_c(\frac{1-n}{24})}{c}I_{s-\frac 12}\pfrac{\pi\sqrt{|n|}}{6c}.
\end{equation*}
It follows that
\begin{equation} \label{eq:a-n-34-neg}
	a(n,\tfrac 34) = \mfrac 1{2\sqrt{|n|}} \sum_{Q\in \Gamma\backslash \sQ_{n}} f(\tau_Q) \qquad \text{ if }n<0.
\end{equation}
For nonsquare $n>0$ we can apply \cite[Proposition~7]{andersen-singular} directly, and we find that
\begin{equation} \label{eq:a-n-34-pos}
	a(n,\tfrac 34) =\mfrac{1}{4\sqrt\pi} \sum_{Q\in \Gamma\backslash \sQ_n} \int_{C_Q} f(\tau) \, \frac{d\tau}{Q(\tau,1)} \qquad \text{ if }n>0.
\end{equation}

\section{The square-indexed coefficients and the proof of Theorem~\ref{thm:h}} \label{sec:square-index}

To describe the square-indexed coefficients, we define the dampened functions $f_Q(\tau)$ which appear in \eqref{eq:fqtil}.
Suppose that $n\equiv 1\pmod{24}$ is  square and let $Q\in \sQ_n$.
Then $Q(x,y)=0$ has two rational roots which correspond to cusps $\mathfrak a_1$ and $\mathfrak a_2$ in $\mathbb{P}^1(\Q)$, and $C_Q$ is defined as the geodesic connecting $\mathfrak a_1$ and $\mathfrak a_2$.
For each $i$ there is a unique $r_i$ and a unique $\gamma_i\in \Gamma_{\infty} \backslash \Gamma$ such that
\[
	\gamma_i W_{r_i} \mathfrak a_i = \infty.
\]
Following the method of \cite{andersen-singular} (see (3.13) of that paper) we define
\begin{equation}\label{eq:fqtildef}
	\tilde f_Q(\tau,s) := \sum_{r\mid 6} \mu(r) \sum_{\substack{\gamma \in \Gamma_\infty \backslash \Gamma \\ \gamma W_r \mathfrak a_i\neq \infty}} \phi_s(\im \gamma W_r \tau) e(-\re\gamma W_r\tau).
\end{equation}
With the notation of \cite{andersen-singular} we have
\[
	P_{1, Q}(\tau,2s-\tfrac 12) = \frac{2\Gamma(s+\frac 14)}{\sqrt{\pi} \Gamma(s-\frac 14)} \tilde f_Q(\tau,2s-\tfrac 12).
\]
Let $\Gamma^*$ be the group generated by $\Gamma$ and the Atkin-Lehner involutions $W_r$ for $r\mid 6$.
Matrices $\gamma=\pmatrix abcd\in \Gamma^*$ act on $Q\in \mathcal Q_n$ by
\[\gamma Q(x,y)=Q(dx-by, -cx+ay).\]
Then for $\gamma\in \Gamma^*$ we have
\begin{equation}\label{eq:gammaQroots}
\gamma\tau_Q=\tau_{\gamma Q}
\end{equation}
and
\begin{equation}\label{eq:gammaQdiff}
\frac{d(\gamma\tau)}{\gamma Q(\gamma\tau,1)}=\frac{d\tau}{Q(\tau,1)}.
\end{equation}
For $\sigma\in \Gamma_0(6)$ we have
\begin{equation} \label{eq:til-f-Q-tran}
	\tilde f_{\sigma Q}(\sigma\tau,s) = \tilde f_Q(\tau,s),
\end{equation}
and, for $r\mid 6$, \eqref{eq:atkincompose} gives
\begin{equation} \label{eq:til-f-Q-W-tran}
	\tilde f_{W_r Q}(W_r\tau,s) = \mu(r)\tilde f_Q(\tau,s).
\end{equation}

As  will be seen in the proof of Proposition~\ref{prop:square-coeffs} below, the function $a(n,s)$ has a pole at $s=\frac 34$ which arises from integrating the constant term in the Fourier expansion of $\tilde f_Q(\tau,s)$.
Motivated by this, we define
\begin{equation}
	f_Q(\tau,s) := \tilde f_Q(\tau,s) - a_s(0) y^{1-s}.
\end{equation}
In particular, with $f_Q(\tau):=f_Q(\tau, 1)$ and $\tilde f_Q(\tau):=\tilde f_Q(\tau,1)$, Lemma~\ref{eq:as0-simple} gives
\begin{equation}\label{eq:fqdef}
	f_Q(\tau) = \tilde f_Q(\tau) - 12.
\end{equation}
The next result gives the evaluation of the coefficients of square index.
\begin{proposition}\label{prop:square-coeffs}
Suppose that $n\equiv 1\pmod{24}$ is  square. Let $\Tr_n(f)$  and $h^*(n)$ be defined by \eqref{eq:fqtil} and \eqref{eq:hurwitz}.
Then
\begin{equation}
	c(s) a(n,s) =  \frac{\chi_{12}(\sqrt n)}{s-\frac 34} +  \frac{2\pi}{\sqrt n}\chi_{12}(\sqrt n) h^*(n)
	 + \frac{\pi}{6} \Tr_n(f) + O\left(s-\tfrac 34\right)
\end{equation}
uniformly for $s\in (\frac 34,1]$.
\end{proposition}

To prove the proposition  we will need the following lemma, which describes a set of representatives for $\Gamma\backslash\sQ_{n}$ when $n$ is  square.
We omit the proof, as it follows along the same lines as the proof of \cite[Lemma~3]{andersen-inf-geo}.

\begin{lemma} \label{lem:square-disc-set}
Suppose that $n=b^2$ with $(b,6)=1$.
Then
\begin{equation}
	\Gamma\backslash \sQ_{n} \cong \big\{W_r [0,b,c] : c\bmod b\big\},
\end{equation}
where
\begin{equation*}
	r=
	\begin{cases}
		1 & \text{ if } b\equiv 1\pmod{12}, \\
		2 & \text{ if } b\equiv 7\pmod{12}, \\
		3 & \text{ if } b\equiv 5\pmod{12}, \\
		6 & \text{ if } b\equiv 11\pmod{12}. \\
	\end{cases}
\end{equation*}
\end{lemma}

\begin{proof}[Proof of Proposition~\ref{prop:square-coeffs}]
From (5.4)  of \cite{andersen-singular}  and \eqref{eq:ansdef} we find that
\begin{equation}
	a(n,s) = \frac{\Gamma(2s)}{2\sqrt\pi \, \Gamma(s-\frac 14)} \sum_{Q\in \Gamma\backslash \sQ_n} \int_{C_Q} \tilde f_Q(\tau,2s-\tfrac 12) \, \frac{d\tau}{Q(\tau,1)}.
\end{equation}
Write $n=b^2$ and let $r\in \{1,2,3,6\}$ be as in Lemma~\ref{lem:square-disc-set}.
By Lemma~\ref{lem:square-disc-set}, \eqref{eq:fqtildef}, \eqref{eq:gammaQroots}, \eqref{eq:til-f-Q-W-tran}, and the fact that $\mu(r)=\chi_{12}(\sqrt n)$ we have
\begin{equation*}
	\sum_{Q\in \Gamma\backslash \sQ_{n}} \int_{C_Q} \tilde f_Q(\tau,2s-\tfrac 12) \, \mfrac{d\tau}{Q(\tau,1)}
	= \chi_{12}(\sqrt n)\sum_{c \bmod b} \int_{-\frac c b}^{-\frac c b + i\infty} \tilde f_{[0,b,c]}(\tau, 2s-\tfrac 12) \, \mfrac{d\tau}{b\tau+c}.
\end{equation*}
Let $g=(b,c)$, write $b=gb'$ and $c=gc'$, and
choose $\gamma_c = \pmatrix{u}{-c'}{6v}{b'} \in \Gamma_0(6)$.
Then $\gamma_c W_6 \infty=-\frac{c'}{b'}$.
We replace $\tau$ by $\gamma_c W_6 \tau$.  Since $\gamma_c W_6[0, -b', v]=[0,b',c']$,
\eqref{eq:gammaQroots} and \eqref{eq:gammaQdiff} give
\begin{equation*}
	\int_{-\frac{c'}{b'}}^{-\frac{c'}{b'}+\frac{i}{b'\sqrt 6}} \tilde f_{[0,b,c]}(\tau,2s-\tfrac 12) \, \mfrac{d\tau}{b\tau+c} = \mfrac{1}{g} \int_{\frac{v}{b'} + \frac{i}{b'\sqrt 6}}^{\frac{v}{b'}+i\infty} \tilde f_{[0,b,-v]}(\tau,2s-\tfrac 12) \, \mfrac{d\tau}{b'\tau-v}.
\end{equation*}
It follows that
\begin{multline} \label{eq:int-ell-n}
	\int_{-\frac c b}^{-\frac c b + i\infty} \tilde f_{[0,b,c]}(\tau, 2s-\tfrac 12) \, \mfrac{d\tau}{b\tau+c} \\
	= \mfrac 1b \int_{\frac1{b'\sqrt 6}}^\infty \left( \tilde f_{[0,b,c]}(-\tfrac {c'}{b'}+i y, 2s-\tfrac 12) + \tilde f_{[0,b,-v]}(\tfrac {v}{b'}+i y, 2s-\tfrac 12) \right) \, \mfrac{dy}{y}.
\end{multline}

The cusps $\mathfrak a_1,\mathfrak a_2$ associated to $Q=[0,b,c]$ are given by $\mathfrak a_1=\infty$ and $\mathfrak a_2=-\frac{c'}{b'}$, so we have
\begin{equation}
	\gamma_1 W_{r_1} = \pmatrix 1001 \quad \text{ and } \quad \gamma_2 W_{r_2} = \pmatrix {w}{*}{b'}{c'},
\end{equation}
for some $w\in\Z$.
Thus, by \eqref{eq:fqtildef} and \eqref{eq:f-tau-s-fourier} we have the Fourier expansion
\begin{multline}
	\tilde{f}_{[0,b,c]}(-\tfrac {c'}{b'}+iy,s)
	= a_s(0) y^{1-s} \\ - \mfrac{2\pi}{\sqrt{b'y}}I_{s-\frac 12}\pmfrac{2\pi}{b'y} e\left(-\mfrac w{b'}\right) + 2\sqrt y \sum_{n\neq 0} a_s(n) K_{s-\frac 12}(2\pi|n|y) e\left(-\mfrac{c'n}{b'}\right).
\end{multline}
The contribution from the constant term $a_s(0)y^{1-s}$ of $\tilde f_{[0,b,*]}(\tau,s)$ to the right-hand side of \eqref{eq:int-ell-n} equals
\begin{equation}\label{eq:maincont}
	\mfrac 2b \, a_{2s-\frac 12}(0) \int_{\frac 1{b'\sqrt 6}}^\infty y^{\frac 32-2s} \, \mfrac{dy}y = \mfrac 1b \, 6^{s-\frac 34}(b')^{2s-\frac 32} a_{2s-\frac 12}(0) \, \frac{1}{s-\frac 34}.
\end{equation}
By the estimates (4.4) and (4.5) of \cite{aa-spt}, which are valid uniformly for $s\in[\frac34,1]$, we have
\begin{equation}
	2\sqrt y \sum_{n\neq 0} a_{2s-\frac 12}(n) K_{2s-1}(2\pi|n|y) e\left(-\mfrac{c'n}{b'}\right) \ll e^{-\pi y} \qquad \text{ as }y\to\infty.
\end{equation}
Then, using that $I_s(y) \ll y^s$ uniformly for $s\in[\frac34,1]$ as $y\to0$, we conclude that
\begin{equation*}
	 f_{[0,b,c]}(-\tfrac{c'}{b'}+iy,2s-\tfrac 12) \ll y^{\frac 12-2s} \qquad \text{ as }y\to\infty.
\end{equation*}
It follows that the contribution from $f_{[0,b,*]}(\tau,s)$ to the right-hand side of \eqref{eq:int-ell-n} converges uniformly for $s\in[\frac 34,1]$.

Therefore
\begin{align*}
	\mfrac 1b \int_{\frac1{b'\sqrt 6}}^\infty &\left(  f_{[0,b,c]}(-\tfrac {c'}{b'}+i y, 2s-\tfrac 12) +  f_{[0,b,-v]}(-\tfrac {v}{b'}+i y, 2s-\tfrac 12) \right) \, \mfrac{dy}{y} \\*
	&= \mfrac 1b \int_{\frac1{b'\sqrt 6}}^\infty \left(  f_{[0,b,c]}(-\tfrac {c'}{b'}+i y) +  f_{[0,b,-v]}(-\tfrac {v}{b'}+i y) \right) \, \mfrac{dy}{y} + O(s-\tfrac 34) \\
	&= \int_{-\frac c b}^{-\frac c b + i\infty}  f_{[0,b,c]}(\tau) \, \mfrac{d\tau}{b\tau+c} + O(s-\tfrac 34),
\end{align*}
where, in the last line, we have reversed the calculations from above.
From \eqref{eq:csdef} and \eqref{eq:as0} we have
\[
	\frac{\Gamma(2s)}{2\sqrt\pi\,\Gamma(s-\frac 14)} 6^{s-\frac 34} c(s) a_{2s-\frac 12}(0) = 1.
\]
With \eqref{eq:maincont} and \eqref{eq:c34} this gives
\begin{equation*}
	c(s) a(n,s) = \frac{\chi_{12}(\sqrt n)}{\sqrt n} \!\! \sum_{c\bmod \sqrt n} \!\! \pfrac{\sqrt n}{(c,\sqrt n)}^{2s-\frac 32} \!\!\! \frac{1}{s-\frac 34} + \frac 1{12} \sum_{Q\in \Gamma\backslash \sQ_{n}} \int_{C_Q} f_Q(\tau) \, \mfrac{d\tau}{Q(\tau,1)} + O(s-\tfrac 34).
\end{equation*}
We have the series expansion $x^{2s-\frac 32} = 1 + 2\log x \,(s-\frac 34)+O(s-\frac 34)^2$.
Using  Lemma~\ref{lem:square-disc-set} and \eqref{eq:hurwitz} we see that for square $n$, we have
\begin{equation} \label{eq:hstar}
	h^*(n) = \frac 1{\pi} \sum_{\ell^2 \mid n} \log \sqrt{\frac {n}{\ell^2}} \sum_{\substack{[a,b,c]\in \Gamma\backslash\sQ_{n/\ell^2} \\ (a,b,c)=1}} \!\!\! 1 = \frac 1{\pi} \sum_{c\bmod \sqrt n} \log \pfrac{\sqrt n}{(c,\sqrt n)}.
\end{equation}
The proposition follows.
\end{proof}

We also require a short lemma.
\begin{lemma} \label{eq:lem-W-deriv-alpha}
If $n>0$ then
\begin{equation}
	\frac{\partial}{\partial s} \sW_n(y,s) \Big|_{s=\frac 34}
	= 4\pi e^{-\frac {\pi n y}{12}} \alpha\left(\mfrac {ny}6\right).
\end{equation}
\end{lemma}

\begin{proof}
By \cite[(9.222.1)]{GR-table} we have the integral representation
\begin{equation*}
	W_{\frac 14,s-\frac 12}(y) = \frac{y^s e^{-\frac y2}}{\Gamma(s-\frac 14)} \int_0^\infty e^{-yt} t^{s-\frac 54} (1+t)^{s-\frac 34} \, dt.
\end{equation*}
Differentiating under the integral sign, we find that
\begin{multline*}
	\frac{\partial}{\partial s} W_{\frac 14,s-\frac 12}(y) \Big|_{s=\frac 34}
	= y^{\frac 14}e^{-\frac y2} \left( \log y - \psi\left(\mfrac 12\right) \right) \\ + \frac{y^{\frac 34}e^{-\frac y2}}{\sqrt \pi} \int_0^\infty e^{-yt} t^{-\frac 12} \log t \, dt
	+ \frac{y^{\frac 34}e^{-\frac y2}}{\sqrt \pi} \int_0^\infty e^{-yt} t^{-\frac 12} \log (1+t) \, dt,
\end{multline*}
where $\psi(z)=\frac{\Gamma'}{\Gamma}(z)$ is the digamma function.
By \cite[(4.352.1)]{GR-table} we have
\[
	\int_0^\infty e^{-yt} t^{-\frac 12} \log t \, dt = \sqrt{\frac \pi y} \, \left( \psi\left(\mfrac 12\right) - \log y \right),
\]
from which the lemma easily follows.
\end{proof}

We now prove Theorem~\ref{thm:h}.

\begin{proof}[Proof of Theorem~\ref{thm:h}]
After Proposition~\ref{prop:hdef}
it remains only to show that $\bh(\tau)$ has Fourier expansion \eqref{eq:h-exp}.
For the nonsquare coefficients, this follows from Proposition~\ref{prop:hdef} and \eqref{eq:a-n-34-neg}--\eqref{eq:a-n-34-pos}.
For square $n>0$, the $n$-th term of $\bh(\tau)$ is given by
\begin{equation} \label{eq:sq-coeff-lim}
	\mfrac 6\pi \lim_{s\to\frac 34^+} \left( c(s)a(n,s)\sW_n(y,s) - \mfrac{\chi_{12}(\sqrt n)}{s-\frac 34} e^{-\frac{\pi n y}{12}} \right) e_{24}(nx).
\end{equation}
By Lemma~\ref{eq:lem-W-deriv-alpha} we have the Taylor expansion
\begin{align*}
	\sW_n(y,s) &= \sW_n(y,\tfrac 34) + \mfrac{\partial}{\partial s} \sW_n(y,s) \big|_{s=\frac 34}(s-\tfrac 34) + O(s-\tfrac 34)^2 \\
		&= e^{-\frac{\pi n y}{12}} \left[ 1 + 4\pi\alpha\pmfrac{ny}{6}(s-\tfrac 34) + O(s-\tfrac 34)^2 \right].
\end{align*}
This, together with Proposition~\ref{prop:square-coeffs}, shows that the expression \eqref{eq:sq-coeff-lim} equals
\begin{equation*}
	\left( \mfrac{12}{\sqrt n} \, \chi_{12}(\sqrt n) h^*(n) + \Tr_n(f)+ 24 \, \chi_{12}(\sqrt n)\alpha\pmfrac{ny}{6}  \right) q^\frac n{24}.
\end{equation*}
Theorem~\ref{thm:h} follows.
\end{proof}

\section{Theorems \ref{thm:BFI} and \ref{thm:innprodlevel4}}\label{sec:DIT}
We briefly sketch how
 Theorem~\ref{thm:BFI} can be deduced from  Theorem~4.2 of \cite{Bruinier:2011}.
Let $L$ be the lattice of Example~2.1 of \cite{Bruinier:2011} with $N=1$, and let $h\in L'/L\cong \Z/2\Z$ denote the non-trivial element.
With $H_{h}(\tau,1)$ as in \cite[Theorem~4.2]{Bruinier:2011} we have the relation
\begin{equation}
	{\bm Z}(\tau) = \tfrac 12\left( H_L(4\tau,1) + H_{L+h}(4\tau,1) \right).
\end{equation}
We note that there are a few errors in \cite{Bruinier:2011} which should be corrected as follows.
First, the term $(\frac 12 \log 2 + \frac 14 \gamma)$ in Lemmas~8.5 and 8.6, and in the definition of $\mathcal F(t)$ in Theorem~4.1 should be changed to $(\log 2+\frac 12\gamma)$ (see (8.10)--(8.12) of \cite{Bruinier:2011}).
With the corrected definition of $\mathcal F(t)$, we have
\[
	\mathcal F(2\sqrt{\pi y} \, m) = -2\pi \alpha(4m^2y).
\]
In particular, Remark~4.3 no longer applies.
Second, the  constant terms in Theorem~4.2 (the first and last lines of the formula for $H_h(\tau,1)$) should be multiplied by $\delta_{h,0}$.

Theorem~\ref{thm:BFI} can also be deduced directly in analogy with  Sections~\ref{sec:def}--\ref{sec:square-index} from the definition of $\bm Z(\tau)$ as a limit.
Since this computation is quite involved, we give a sketch here.
Let
\begin{equation}
	c'(s) := \frac{2^{4s-1}\Gamma(s+\frac 14)\Gamma(s-\frac 14)^2 \zeta(2s-\frac 12)\zeta(4s-1)}{\pi^{s+\frac 34}\Gamma(2s-1)\zeta(4s-2)},
\end{equation}
and note that  $c'(\frac 34)=\frac{4\pi}{3}$.
With $P_0^+(\tau, s)$ as in Section~5 of \cite{DIT:CycleIntegrals}
we define
\begin{equation}
	\bm Z(\tau) := \frac{1}{4\pi} \lim_{s\to \frac 34} \left( c'(s) P_0^+(\tau,s) - \frac{\theta(\tau)}{s-\frac 34} \right).
\end{equation}
By (2.24) of \cite{DIT:CycleIntegrals}, the contribution from the constant term of $P_0^+(\tau,s)$ equals
\begin{equation}
	\mfrac{1}{4\pi} \lim_{s\to\frac 34} \left(   \frac{2^{2s-\frac 12}c'(s)b_0(0,s)}{(2s-1)\Gamma(2s-\frac 12)}y^{\frac 34-s} - \frac{1}{s-\frac 34}   \right) = \frac{\gamma-\log(16\pi y)}{4\pi}.
\end{equation}
For $d>0$ let $\fQ_d := \{[a,b,c]: b^2-4ac=d\}$ and let $\Gamma_1:=\PSL_2(\Z)$.
Let $b_0(d,s)$ denote the $d$-th coefficient of $P_0^+(\tau,s)$ (see \cite[(2.20--21)]{DIT:CycleIntegrals}).
For non-square $d$, the function $b_0(d,s)$ is analytic at $s=3/4$, so the coefficients of non-square index in ${\bm Z}(\tau)$ agree with the corresponding coefficients of $\widehat{\bm Z}_+(\tau)$.

Suppose that $d>0$ is a square.
By (4.5) of \cite{andersen-inf-geo} we have
\begin{equation} \label{eq:b-0-d-s-1}
	b_0(d,\tfrac s2+\tfrac 14) = \frac{2^{1-2s}\pi^{\frac{s+1}2}\Gamma(s)}{\zeta(s)\Gamma(\frac s2)^2} \sum_{Q\in \Gamma_1\backslash \fQ_d} \int_{C_Q} G_{0,Q}(\tau,s) \, \frac{\sqrt d \, d\tau}{Q(\tau,1)},
\end{equation}
where $G_{0,Q}(\tau,s)$ is a dampened version of the Eisenstein series $G_0(\tau,s)$ defined in \S 4 of \cite{DIT:CycleIntegrals}.
As in Lemma~\ref{lem:square-disc-set} we have
\begin{equation} \label{eq:fQ-square-reps}
	\Gamma_1\backslash \fQ_d \cong \{[0,b,c] : c\bmod b\} \qquad \text{ for $d=b^2$}.
\end{equation}
Following the proof of Proposition~\ref{prop:square-coeffs} above, we find that
\begin{equation} \label{eq:b-0-d-s-2}
	b_0(d,\tfrac s2+\tfrac 14) = \frac{2^{2-2s}\pi^{\frac{s+1}2} \Gamma(s) c_0(0,s)}{\zeta(s)\Gamma(\frac s2)^2} \sum_{c\bmod b} \pfrac{b}{(b,c)}^{s-1} \frac{1}{s-1} + O(s-1),
\end{equation}
where $c_0(0,s)=\frac{\sqrt \pi \, \Gamma(s-1/2)\zeta(2s-1)}{\Gamma(s)\zeta(2s)}$ is the coefficient of $y^{1-s}$ in $G_0(\tau,s)$.
We replace $s$ by $2s-\frac 12$ in \eqref{eq:b-0-d-s-2} and multiply by $c'(s)$.
Since
\[
	c'(s) \times \frac{2^{2-4s}\pi^{s+\frac 14} \Gamma(2s-\frac 12) c_0(0,2s-\frac 12)}{\zeta(2s-\frac12)\Gamma(s-\frac 14)^2} = 2 \, \Gamma(s+\tfrac 14),
\]
we find, using \eqref{eq:hstar}, that
\[
	c'(s) b_0(d,s) = \frac{2\sqrt d \, \Gamma(s+\tfrac 14)}{s-\frac 34} + 4\pi\Gamma(s+\tfrac 14)h^*(d) + O(s-\tfrac 34).
\]
Using the Taylor expansion
\begin{align*}
	(4\pi d y)^{-\frac 14} W_{\frac 14,s-\frac 12}(4\pi|d|y) e(dx)
	= q^d + 4\pi \alpha(4dy)q^d (s-\tfrac 34) + O\big((s-\tfrac 34)^2\big),
\end{align*}
we find that the $d$-th term in the Fourier expansion of ${\bm Z}(\tau)$ equals
\begin{multline*}
	\mfrac{1}{4\pi} \lim_{s\to\frac 34} \left( c'(s) b_0(d,s) d^{-\frac 12} \Gamma(s+\tfrac 14)^{-1} (4\pi y)^{-\frac 14} W_{\frac 14,s-\frac 12}(4\pi dy) e(dx) - \frac{2q^d}{s-\frac 34}\right)  \\
	=2\alpha(4dy)q^d + d^{-\frac 12}  h^*(d)q^d.
\end{multline*}
Theorem~\ref{thm:BFI} follows.

We turn to the proof of Theorem~\ref{thm:innprodlevel4}.
Recall  (see \cite{ZagierT} for example) that for each positive discriminant $d$  there exists a unique weight $\frac{3}{2}$ weakly holomorphic modular form $g_d$ on $\Gamma_0(4)$ of the form
\begin{equation}\label{eq:gd_exp}
g_d(\tau)= q^{-d}+\sum_{0\le n\equiv 0,3 (4)}B(d,n)q^n,
\end{equation}
where the $B(d,n)$ are integers and $B(d,0)=-2\delta_{\square}(d).$
The first three forms are
\[\begin{aligned}
g_1&=q^{-1}-2+248\, q^3-492\,q^4+4119\, q^7+\dots,\\
g_4&=q^{-4}-2-26752\,q^3-143376\,q^4-8288256\,q^7+\dots,\\
g_5&=q^{-5}+0+85995\,q^3-565760\,q^4+52756480\,q^7+\dots.\\
\end{aligned}\]
To follow the notation in \cite{Duke:2011a}, we define
\[g_0(\tau):=\widehat{\bm Z}_-(\tau).\]

Duke, Imamo\={g}lu and T\'{o}th \cite[Prop. 4.1]{Duke:2011a} proved that for any positive non-square discriminant $d$ we have
\begin{equation*}
\langle g_d,g_0\rangle_{\operatorname{reg}} = -\frac{3}{4}\frac{h^*(d)}{\sqrt{d}},
\end{equation*}
where
\[\langle g_d,g_0\rangle_{\operatorname{reg}}:=\lim_{Y\rightarrow\infty} \int_{\mathcal{F}_Y(4)} g_d(\tau)\overline{g_0(\tau)}y^\frac32\frac{dxdy}{y^2}\]
and
$\mathcal{F}_Y(4)$ is a  fundamental domain for $\Gamma_0(4)$ truncated by removing $Y$-neighborhoods of the cusps.

To define an inner product in the case when $d$ is square, we adopt the strategy in a recent paper of Bringmann, Diamantis, and Ehlen \cite{BDE}.
In that paper a general regularization of the inner product of two weakly holomorphic modular forms was given by introducing extra terms
in two auxiliary variables.  In this section and the next, we will adapt this strategy in only the generality we need; in particular we require only one of the
auxiliary variables for each application.

The language of vector valued forms is most convenient in this section. Suppose that $f(\tau)=\sum\limits_{n\equiv 0, 3 (4)} a(n)q^n$ is a weakly holomorphic modular form of weight $3/2$ on  $\Gamma_0(4)$ in the plus-space, and write $f(\tau)=f^e(\tau)+f^o(\tau)$, where these denote the sums over even and odd indices respectively.
To $f(\tau)$ we associate the vector valued form
\[\vec f(\tau):=f^e(\tau/4)\mathfrak e_0+f^o(\tau/4)\mathfrak e_1\]
Then $\vec f$ transforms in weight $3/2$ with respect to a certain representation of $\Mp_2(\Z)$ (see e.g. \cite[\S4.3]{BDE} for details).
Similarly, we  write
\[\vec g_d(\tau):=g_d^e(\tau/4)\mathfrak e_0+g_d^o(\tau/4)\mathfrak e_1\qquad\text{and}\qquad
\vec {\bm{Z}}(\tau):=\bm{Z}^e(\tau/4)\mathfrak e_0+\bm{Z}^o(\tau/4)\mathfrak e_1.\]

Suppose that $d>0$ is not square.  Let $\mathcal F_Y$ be the standard fundamental domain for $\SL_2(\Z)$, truncated at height $Y$.
A computation as in \cite[\S 4.3]{BDE}, using Lemma~3.2 of \cite{Duke:2011a} shows that
\begin{equation}\label{eq:inn_prod_rel}
\langle g_d,g_0\rangle_{\operatorname{reg}}=
\frac{3}{4}\lim_{Y\rightarrow\infty} \int_{\mathcal{F}_Y}\vec g_d(\tau)\cdot\overline{\vec g_0(\tau)}y^\frac32\frac{dxdy}{y^2}.
\end{equation}
(We note that the constant $\frac34$ appears incorrectly as $\frac32$ in the corresponding computation of \cite[\S 4.3]{BDE}; this arises from  the fact that the relationship
$\xi_\frac12\bm G=\bm g$ in that section should read $\xi_\frac12  \left( 2\bm G \right) =\bm g$.)

For any $d>0$,  we define
\begin{equation}\label{eq:Igdg0}
I(g_d,g_0;s):=\lim_{Y\rightarrow\infty} \int_{\mathcal{F}_Y}\vec g_d(\tau)\cdot\overline{\vec g_0(\tau)}y^{\frac 32-s}\frac{dxdy}{y^2}.
\end{equation}
We will show that $I(g_d, g_0; s)$ is defined for $\re (s)$ sufficiently large, and that it has a meromorphic continuation to a neighborhood of $s=0$.
We may therefore define the extended inner product
\begin{equation}\label{eq:innproddef}
\langle g_d,g_0\rangle_4 :=\underset{s=0}{\CT}(I(g_d,g_0;s))
\end{equation}
as the constant term in the Laurent expansion at $s=0$.  By \eqref{eq:inn_prod_rel}
we have
\[\langle g_d,g_0\rangle_4 =\frac43\,\langle g_d,g_0\rangle_{\operatorname{reg}}\quad\quad\text{if $d$ is not square;}\]
this also follows from the computations below.

To show that the definition makes sense,  we truncate at $y=1$ to obtain
\begin{equation*}
\int_{\mathcal{F}_Y}\vec g_d(\tau)\cdot\overline{\vec g_0(\tau)}y^{\frac 32-s}\frac{dxdy}{y^2} = \int_{\mathcal{F}_1}\vec g_d(\tau)\cdot\overline{\vec g_0(\tau)}y^{\frac 32-s}\frac{dxdy}{y^2}+\int_1^Y \int_{-\frac12}^{\frac12}\vec g_d(\tau)\cdot\overline{\vec g_0(\tau)}y^{\frac 32-s}\frac{dxdy}{y^2}.
\end{equation*}
Using  Fourier expansions  and integrating term by term,  the second integral becomes
\begin{multline}\label{eq:inn_prod_justify}
\int_1^Y \sum_{n>0} H(n)B(d, n)e^{- \pi n y}y^{-\frac12-s}\,  dy
-\frac{\delta_\square(d)\sqrt{d}}4\int_1^Y\beta_\frac32(\pi d y)e^{\pi d y}y^{-\frac12-s}\,  dy\\
+\delta_\square(d)\int_1^Y \left( \mfrac16-\mfrac1{2\pi\sqrt y}  \right) y^{-\frac12-s}\,  dy.
\end{multline}
If $\{c(n)\}$ are the coefficients of a weakly holomorphic modular form or a mock modular form, then by
\cite[Lemma 3.4]{Bruinier:2004} we have the estimate
\begin{equation}\label{eq:coeff_est}
c(n)\ll e^{C\sqrt{n}}\ \ \text{for some  $C$ as $n\to\infty$}.
\end{equation}
By \cite[(8.11.2)]{NIST:DLMF} we have
\begin{equation}\label{eq:beta_est}
 \beta_k(y)\ll y^{-k}e^{-y} \ \ \text{as $y\to\infty$.}
\end{equation}
We also have  the crude estimate $H(n)\ll n^{1+\epsilon}$.

It follows that the integral defining $I(g_d, g_0; s)$ converges for  $\re (s)>\frac 12$.
In the region of convergence we have
\begin{multline}\label{eq:innprodcomp}
I(g_d, g_0; s)=
\lim_{Y\rightarrow\infty} \left[\int_{\mathcal{F}_Y}\vec g_d(\tau)\cdot\overline{\vec g_0(\tau)}y^{\frac 32-s}\frac{dxdy}{y^2}
-\delta_\square(d)\int_1^Y \left( \mfrac16-\mfrac1{2\pi\sqrt y}  \right) y^{-\frac12-s}\,  dy\right]\\
+\delta_\square(d)\int_1^\infty \left( \mfrac16-\mfrac1{2\pi\sqrt y}  \right) y^{-\frac12-s}\,  dy.
\end{multline}
By the discussion above, the first term is holomorphic in a neighborhood of $s=0$. The second term has a meromorphic continuation to $s=0$.
This justifies the definition \eqref{eq:innproddef}, and shows that we have
\begin{equation}\label{E:def2}
\langle g_d,g_0\rangle_4 = \lim_{Y\rightarrow\infty}\left[\int_{\mathcal{F}_Y}\vec g_d(\tau)\cdot\overline{\vec g_0(\tau)}y^\frac{3}{2}\frac{dxdy}{y^2}
-\delta_\square(d) \left( \mfrac{\sqrt{Y}-1}{3}-\mfrac{\log Y}{2\pi} \right) 
\right]
-\frac{\delta_\square(d)}3.
\end{equation}

Since $\xi_{\frac{1}{2}}{\bm Z}=-2g_0$, we have
\[
\xi_{\frac{1}{2}}{\bm Z}^{\rm e}(\tau/4)=-g_0^{\rm e}(\tau/4), \qquad \xi_{\frac{1}{2}}{\bm Z}^{\rm o}(\tau/4)=-g_0^{\rm o}(\tau/4).\]
If $g$ is holomorphic, then by Stokes' theorem (see, e.g., \cite[Lemma 3.2]{Duke:2011a}, noting that the identity there should read
$d\tau d\overline\tau=-2i\,dx dy$)
 we have
\begin{equation}\label{eq:stokes}
\int_{\mathcal F_Y}g(\tau)h_{\overline\tau}(\tau)\, d\tau d\overline\tau=-\int_{\partial\mathcal{F}_Y} g(\tau)h(\tau)\, d\tau.
\end{equation}
Therefore
\begin{align*}
\int_{\mathcal{F}_Y}\vec g_d(\tau)\cdot\overline{\vec g_0(\tau)}y^\frac{3}{2}\frac{dxdy}{y^2}
&=-\int_{\mathcal{F}_Y}\left(g_d^{\rm e}\left(\mfrac{\tau}{4}\right)\overline{\xi_{\frac{1}{2}}{\bm Z}^{\rm e}\left(\mfrac{\tau}{4}\right)}+g_d^{\rm o}\left(\mfrac{\tau}{4}\right)
\overline{\xi_{\frac{1}{2}}{\bm Z}^{\rm o}\left(\mfrac{\tau}{4}\right)}\right)y^{-\frac{1}{2}}dxdy\\
&=\int_{\partial\mathcal{F}_Y}\left(g_d^{\rm e}\left(\mfrac{\tau}{4}\right){\bm Z}^{\rm e}\left(\mfrac{\tau}{4}\right)+g_d^{\rm o}\left(\mfrac{\tau}{4}\right){\bm Z}^{\rm o}\left(\mfrac{\tau}{4}\right)
\right)d\tau\\
&=-\int_{-\frac{1}{2}+iY}^{\frac{1}{2}+iY}\left(g_d^{\rm e}\left(\mfrac{\tau}{4}\right){\bm Z}^{\rm e}\left(\mfrac{\tau}{4}\right)+g_d^{\rm o}\left(\mfrac{\tau}
{4}\right){\bm Z}^{\rm o}\left(\mfrac{\tau}{4}\right)\right)d\tau.
\end{align*}
Using the Fourier expansions \eqref{eq:BFI} and \eqref{eq:gd_exp}, we find that
 \begin{multline}\label{eq:up_to_Y}
\int_{\mathcal{F}_Y}\vec g_d(\tau)\cdot\overline{\vec g_0(\tau)}y^\frac{3}{2}\frac{dxdy}{y^2}\\
=
-\frac{h^*(d)}{\sqrt{d}}+\delta_\square(d) \left( \mfrac{\sqrt{Y}}{3}+\mfrac{\gamma-\log 4\pi Y}{2\pi}-\alpha(dY) \right) 
 -\sum_{n>0}B(d,n)\frac{h^*(-n)}{\sqrt{n}}\beta_\frac12(\pi nY).
\end{multline}

Since  $e^{-x}\le x^{-1}$  for all $x>0$, we have $\alpha(Y)\to0$ as $Y\to\infty$.
By  \eqref{E:def2}, \eqref{eq:up_to_Y},  \eqref{eq:coeff_est}, and \eqref{eq:beta_est},
we have
\[\langle g_d,g_0\rangle_4=
-\frac{h^*(d)}{\sqrt{d}}+\delta_\square(d) \frac{\gamma-\log 4\pi}{2\pi},\]
and Theorem~\ref{thm:innprodlevel4} follows.

\section{Regularized Inner products for $\SL_2(\Z)$} \label{sec:inn-prod}
Here we prove an analogue of Theorem~\ref{thm:innprodlevel4} for $\SL_2(\Z)$.
Let $F$ be the harmonic Maass form of weight $3/2$ and multiplier $\bar\chi$ on $\SL_2(\Z)$ defined in \eqref{eq:fdef}.  We also introduce a natural infinite family of weakly holomorphic modular forms.

\begin{lemma}\label{L:gd}
For any integer $d>1$ with $d\equiv 1 \mod 24$, there exists a unique weight $\frac{3}{2}$ weakly holomorphic modular form $h_d$ on $\SL_2(\Z)$ with multiplier $\bar{\chi}$ such that
\[h_d(\tau) =q^{-\frac{d}{24}}+\sum_{-1\le n \equiv 23 (24)}A(d,n)q^{\frac{n}{24}} .\]
Furthermore, we have $A(d,-1)=- \chi_{12}(\sqrt{d}).$
\end{lemma}

\begin{proof}
Leting $j(\tau)$ be the usual $j$-invariant and defining
\[\Theta:=\frac{1}{2\pi i}\frac{d}{d\tau}=q\frac{d}{dq},\]
we have
\begin{equation*}
h_{25}(\tau)=-\frac{\Theta(j(\tau))}{\eta(\tau)}=q^{-\frac{25}{24}}+q^{-\frac{1}{24}}-196882q^{\frac{23}{24}}-\cdots.\\
\end{equation*}
The subsequent forms $h_d(\tau)$ are constructed by multiplying $h_{25}(\tau)$ by a suitable element of $\mathbb{C}[j].$ For instance, the next two forms are
\begin{align*}
h_{49}(\tau)&=(j(\tau)-745)h_{25}(\tau) = q^{-\frac{49}{24}}+q^{-\frac{1}{24}}-21296875q^{\frac{23}{24}}-\cdots,\\
h_{73}(\tau)&= (j(\tau)^2 - 1489j(\tau) + 357395)h_{25}(\tau) =q^{-\frac{73}{24}}-842609326q^{\frac{23}{24}}-\cdots.
\end{align*}
The remaining claim follows from the fact that each $h_d(\tau)\eta(\tau)$ is a weakly holomorphic modular form of weight $2$ on $\SL_2(\Z)$.
\end{proof}

We  define in \eqref{eq:IPdef} an inner product $\langle\cdot, \cdot\rangle_1$ which extends the natural regularized inner product, and
we will prove
\begin{theorem}\label{T:RIP}
Let $f$ be the modular function on $\Gamma_0(6)$ defined in \eqref{E:f}.
Then for each positive integer $d\equiv 1 \bmod 24$, we have
\begin{equation*}
\sqrt{24}\,\langle h_d, F\rangle_1 =- \Tr_d(f)+\chi_{12}(\sqrt{d})\left(\Tr_1(f)-12 \frac{h^*(d)}{\sqrt{d}}-i\right).
\end{equation*}
\end{theorem}

The usual  regularized inner product of $h_d$ and $F$ is given by
\[\langle h_d, F\rangle_{\operatorname{reg}} =\lim_{Y\rightarrow \infty}\int_{\mathcal{F}_Y} h_d(\tau)\overline{F(\tau)}y^\frac32 \frac{dx dy}{y^2}.\]
The computation below will show that $\langle h_d, F\rangle_{\operatorname{reg}}$ exists if and only if $d$ is not square, and that
\[\sqrt{24}\,\langle h_d, F\rangle_{\operatorname{reg}}=-\Tr_d(f)\quad\quad\text{if $d$ is not square}.\]
We again adopt the strategy from \cite{BDE} to extend the inner product to the case when $d$ is square.  Since there are exponentially growing terms in this case (see   below) the extension differs from that of the last section.

Define
\begin{equation*}
I(h_d,F;w):=\int_{\mathcal{F}}h_d(\tau)\overline{F(\tau)}y^{\frac{3}{2}}e^{-wy} \frac{dxdy}{y^2}.
\end{equation*}
The integral converges when $\re w\gg 0$.
We will show that there is an analytic continuation to $w=0$; we then define
\begin{equation}\label{eq:IPdef}
\langle h_d,F\rangle_1 := I(h_d,F;0).
\end{equation}

For $\re w\gg 0$, we have
\begin{equation*}
I(h_d,F;w)=\lim_{Y\rightarrow \infty}\left(\int_{\mathcal{F}_1}h_d(\tau)\overline{F(\tau)}y^{-\frac{1}{2}}e^{-wy}dxdy +\int_{1}^Y\int_{-\frac12}^\frac12 h_d(\tau)\overline{F(\tau)}y^{-\frac{1}{2}}e^{-wy}dxdy \right).
\end{equation*}
Integrating
 term by term (note that $s(0)=-\frac1{12}$) yields
\begin{multline*}
\int_1^Y\int_{-\frac12}^\frac12 h_d(\tau)\overline{F(\tau)}y^{-\frac{1}{2}}e^{-wy}dxdy \\= \frac{\chi_{12}(\sqrt{d})}{12}\int_1^Y y^{-\frac{1}{2}}e^{(\frac{\pi}{6}-w)y}dy -\frac{\chi_{12}(\sqrt{d})\sqrt{d}}{2}\int_1^Y y^{-\frac{1}{2}}e^{(\frac{d\pi}{6}-w)y}\beta_{\frac{3}{2}}\pmfrac{\pi dy}6 dy \\
 + \frac{\chi_{12}(\sqrt{d})}{2}\int_1^Y y^{-\frac{1}{2}}e^{(\frac{\pi}{6}-w)y}\beta_{\frac{3}{2}}\pmfrac{\pi y}6 dy
 +\sum_{n>0} A(d,n)s\pmfrac{n+1}{24}\int_1^Y y^{-\frac{1}{2}}e^{(-\frac{\pi n}{6}-w)y}dy.
\end{multline*}
By \eqref{eq:coeff_est} and \eqref{eq:beta_est},
we see that all but the first integral on the right side converge absolutely on $\re w\ge 0$ as $Y\rightarrow\infty$.

For $\re w\gg 0$, we have
\begin{align*}
I(h_d,F;w)&=\lim_{Y\rightarrow \infty}\left(\int_{\mathcal{F}_Y}h_d(\tau)\overline{F(\tau)}y^\frac{3}{2}e^{-wy} \frac{dxdy}{y^2}-\frac{\chi_{12}(\sqrt{d})}{12}\int_1^Y y^{-\frac{1}{2}}e^{(\frac{\pi}{6}-w)y}dy\right)\\
&\qquad +\frac{\chi_{12}(\sqrt{d})}{12} \int_1^\infty y^{-\frac12}e^{(\frac{\pi}{6}-w)y}dy.
\end{align*}
Using \eqref{eq:betadef}, the last term is
\[\frac{\chi_{12}(\sqrt{d})}{12}\frac{\sqrt\pi}{\sqrt{w-\frac\pi6}} \beta_\frac12 \left( w-\frac\pi6 \right) ,\]
so   we have
\begin{multline}\label{E:RIP2}
\langle h_d,F\rangle_1=I(h_d,F;0)=\lim_{Y\rightarrow \infty}\left(\int_{\mathcal{F}_Y}h_d(\tau)\overline{F(\tau)}y^\frac{3}{2} \frac{dxdy}{y^2}-\frac{\chi_{12}(\sqrt{d})}{12}\int_1^Y y^{-\frac{1}{2}}e^{\frac{\pi y}{6}}dy\right)
\\-i\, \frac{\chi_{12}(\sqrt{d})}{\sqrt{24}} \beta_{\frac{1}{2}}\left(- \mfrac{\pi}6\right).
\end{multline}

We turn to the proof of Theorem~\ref{T:RIP}.
Since $-\sqrt{24}F(\tau)=\xi_{\frac{1}{2}}\bh(\tau)$, arguing as above using
\eqref{eq:stokes} gives

\begin{equation*}
\int_{\mathcal{F}_Y} h_d(\tau)\overline{F(\tau)}y^{\frac{3}{2}} \frac{dx dy}{y^2}
= \frac{1}{\sqrt{24}}\int_{\partial \mathcal{F}_Y}h_d(\tau)\bh(\tau) d\tau
= \frac{-1}{\sqrt{24}}\int_{-\frac{1}{2}+iY}^{\frac{1}{2}+iY}h_d(\tau)\bh(\tau) d\tau.
\end{equation*}
Integrating term by term gives
\begin{multline}\label{E:IP1}
\sqrt{24}\int_{\mathcal{F}_Y} h_d(\tau)\overline{F(\tau)}y^\frac32 \frac{dx dy}{y^2} =
-\Tr_d(f)+\chi_{12}(\sqrt d ) \left(  \Tr_1(f)-12\mfrac{h^*(d)}{\sqrt d} \right)  \\*
+i\chi_{12}(\sqrt{d}) \left( \beta_{\frac 12}\pmfrac{-\pi Y}6-1 \right) -24\chi_{12}(\sqrt{d}) \left( \alpha\pmfrac{dY}6-\alpha\pmfrac Y6 \right)  \\
-\sum_{n>0}A(d,n)\mfrac{\Tr_{-n}(f)}{\sqrt{n}}\beta_{\frac 12}\pmfrac{\pi nY}6.
\end{multline}
Using  \eqref{eq:beta_to_gamma} and \eqref{eq:gammastardef}, we find that
\begin{equation}\label{eq:betarel}
\beta_{\frac 12}\pmfrac{-\pi Y}6-1
= \beta_{\frac12}\pmfrac{-\pi}{6}-1 - \frac{i}{\sqrt{6}}\int_1^Y y^{-\frac12} e^{\frac{\pi y}{6}}dy.
\end{equation}
Theorem~\ref{T:RIP} follows from \eqref{E:RIP2}, \eqref{E:IP1}, and \eqref{eq:betarel}.

In closing, we remark that by generalizing these arguments it would be possible to investigate inner products of larger families of forms
in the spirit of \cite{BDE} and  \cite{Duke:2011a}.

\bibliographystyle{plain}
\bibliography{polyharmonic}

\begin{thebibliography}{10}

\bibitem{aa-52}
Scott Ahlgren and Nickolas Andersen.
\newblock Weak harmonic {M}aass forms of weight 5/2 and a mock modular form for
  the partition function.
\newblock {\em Res. Number Theory}, 1:1:10, 2015.

\bibitem{aa-spt}
Scott Ahlgren and Nickolas Andersen.
\newblock Algebraic and transcendental formulas for the smallest parts
  function.
\newblock {\em Adv. Math.}, 289:411--437, 2016.

\bibitem{andersen-inf-geo}
Nickolas Andersen.
\newblock Periods of the {$j$}-function along infinite geodesics and mock
  modular forms.
\newblock {\em Bull. Lond. Math. Soc.}, 47(3):407--417, 2015.

\bibitem{andersen-singular}
Nickolas Andersen.
\newblock Singular invariants and coefficients of harmonic weak {M}aass forms
  of weight 5/2.
\newblock {\em Forum Math.}, 29(1):7--29, 2017.

\bibitem{Andrews:spt}
George~E. Andrews.
\newblock The number of smallest parts in the partitions of {$n$}.
\newblock {\em J. Reine Angew. Math.}, 624:133--142, 2008.

\bibitem{Bringmann:duke}
Kathrin Bringmann.
\newblock On the explicit construction of higher deformations of partition
  statistics.
\newblock {\em Duke Math. J.}, 144(2):195--233, 2008.

\bibitem{BDE}
Kathrin Bringmann, Nikolaos Diamantis, and Stephan Ehlen.
\newblock Regularized inner products and errors of modularity.
\newblock {\em Int. Math. Res. Not. IMRN}, (24):7420--7458, 2017.

\bibitem{hmf-book}
Kathrin Bringmann, Amanda Folsom, Ken Ono, and Larry Rolen.
\newblock {\em Harmonic {M}aass forms and mock modular forms: theory and
  applications}, volume~64 of {\em American Mathematical Society Colloquium
  Publications}.
\newblock American Mathematical Society, Providence, RI, 2017.

\bibitem{Bruinier:2004}
J.~H. Bruinier and J.~Funke.
\newblock On two geometric theta lifts.
\newblock {\em Duke Math. J.}, 125(1):45--90, 2004.

\bibitem{Bruinier:2011}
Jan~H. Bruinier, Jens Funke, and \"Ozlem Imamo{\=g}lu.
\newblock Regularized theta liftings and periods of modular functions.
\newblock {\em J. Reine Angew. Math.}, 703:43--93, 2015.

\bibitem{bf-traces}
Jan~Hendrik Bruinier and Jens Funke.
\newblock Traces of {CM} values of modular functions.
\newblock {\em J. Reine Angew. Math.}, 594:1--33, 2006.

\bibitem{NIST:DLMF}
{\it NIST Digital Library of Mathematical Functions}.
\newblock http://dlmf.nist.gov/, Release 1.0.15 of 2017-06-01.
\newblock F.~W.~J. Olver, A.~B. {Olde Daalhuis}, D.~W. Lozier, B.~I. Schneider,
  R.~F. Boisvert, C.~W. Clark, B.~R. Miller and B.~V. Saunders, eds.

\bibitem{DIT:CycleIntegrals}
W.~Duke, {{\"O}}. Imamo{\=g}lu, and {{\'A}}. T{{\'o}}th.
\newblock Cycle integrals of the {$j$}-function and mock modular forms.
\newblock {\em Ann. of Math. (2)}, 173(2):947--981, 2011.

\bibitem{Duke:2011a}
W.~Duke, {{\"O}}. Imamo{\=g}lu, and {{\'A}}. T{{\'o}}th.
\newblock Real quadratic analogs of traces of singular moduli.
\newblock {\em Int. Math. Res. Not.}, (13):3082--3094, 2011.

\bibitem{FO:spt}
Amanda Folsom and Ken Ono.
\newblock The {$spt$}-function of {A}ndrews.
\newblock {\em Proc. Natl. Acad. Sci. USA}, 105(51):20152--20156, 2008.

\bibitem{Garvan:spt}
F.~G. Garvan.
\newblock Congruences for {A}ndrews' spt-function modulo powers of {$5$}, {$7$}
  and {$13$}.
\newblock {\em Trans. Amer. Math. Soc.}, 364(9):4847--4873, 2012.

\bibitem{GR-table}
I.~S. Gradshteyn and I.~M. Ryzhik.
\newblock {\em Table of integrals, series, and products}.
\newblock Elsevier/Academic Press, Amsterdam, eighth edition, 2015.
\newblock Translated from the Russian, Translation edited and with a preface by
  Daniel Zwillinger and Victor Moll, Revised from the seventh edition
  [MR2360010].

\bibitem{GKZ}
B.~Gross, W.~Kohnen, and D.~Zagier.
\newblock Heegner points and derivatives of {$L$}-series. {II}.
\newblock {\em Math. Ann.}, 278(1-4):497--562, 1987.

\bibitem{LR-polyharmonic}
Jeffrey~C. Lagarias and Robert~C. Rhoades.
\newblock Polyharmonic {M}aass forms for {PSL}$(2,\mathbb{Z})$.
\newblock {\em The Ramanujan Journal}, pages 1--42, 2016.

\bibitem{Lehmer:series}
D.~H. Lehmer.
\newblock On the series for the partition function.
\newblock {\em Trans. Amer. Math. Soc.}, 43(2):271--295, 1938.

\bibitem{pribitkin}
Wladimir de~Azevedo Pribitkin.
\newblock A generalization of the {G}oldfeld-{S}arnak estimate on {S}elberg's
  {K}loosterman zeta-function.
\newblock {\em Forum Math.}, 12(4):449--459, 2000.

\bibitem{zagier-class-numbers}
Don Zagier.
\newblock Nombres de classes et formes modulaires de poids {$3/2$}.
\newblock {\em C. R. Acad. Sci. Paris S{\'e}r. A-B}, 281(21):Ai, A883--A886,
  1975.

\bibitem{ZagierT}
Don Zagier.
\newblock Traces of singular moduli.
\newblock In {\em Motives, polylogarithms and {H}odge theory, {P}art {I}
  ({I}rvine, {CA}, 1998)}, volume~3 of {\em Int. Press Lect. Ser.}, pages
  211--244. Int. Press, Somerville, MA, 2002.

\end{thebibliography}

\end{document}